\documentclass[11pt]{article}

\usepackage{amssymb,amsmath,amsfonts,amsthm}
\usepackage{latexsym}
\usepackage{graphics}
\usepackage{indentfirst}
\usepackage{tikz}
\usepackage{mathtools}

\DeclarePairedDelimiter\floor{\lfloor}{\rfloor}

\usetikzlibrary{shapes,arrows}

\usepackage{hyperref}
\usepackage{comment}
\usepackage{bm}

\allowdisplaybreaks

\newtheorem{innercustomthm}{Theorem}
\newenvironment{customthm}[1]
  {\renewcommand\theinnercustomthm{#1}\innercustomthm}
  {\endinnercustomthm}

\newtheorem*{thm*}{Theorem}
\newtheorem{thm}{Theorem}
\newtheorem{lem}[thm]{Lemma}
\newtheorem{pro}[thm]{Proposition}

\newtheorem{cor}[thm]{Corollary}

\newtheorem{ques}[thm]{Question}

\newcommand{\N}{\mathbb{N}}

\newcommand{\col}{\mathrm{col}}

\newenvironment{graph}[1][scale=1]{
\begin{tikzpicture}[#1]
\tikzstyle{vertex}=[circle, draw, fill, inner sep=0pt, minimum size=4pt]%
\tikzstyle{every path}=[line width=0.5pt]%
\tikzstyle{G}=[dashed]%
\tikzstyle{F}=[solid]
}{\end{tikzpicture}}

\begin{document}
\title{DP-Coloring Cartesian Products of Graphs}
\author{Hemanshu Kaul$^1$,
Jeffrey A. Mudrock$^2$,
Gunjan Sharma$^1$,
Quinn Stratton$^1$
}

\footnotetext[1]{Department of Applied Mathematics, Illinois Institute of Technology, Chicago, IL 60616. E-mail: {\tt kaul@iit.edu},  {\tt gsharma7@hawk.iit.edu}, {\tt qstratto@hawk.iit.edu} }

\footnotetext[2]{Department of Mathematics, College of Lake County, Grayslake, IL 60030. E-mail: {\tt jmudrock@clcillinois.edu}}

\date{}

\maketitle 
\begin{abstract}
DP-coloring (also called correspondence coloring) is a generalization of list coloring introduced by  Dvo\v{r}\'{a}k and Postle in 2015. Motivated by results related to list coloring Cartesian products of graphs, we initiate the study of the DP-chromatic number, $\chi_{DP}$, of the same.
We show that $\chi_{DP}(G \square H) \leq \text{min}\{\chi_{DP}(G) + \text{col}(H), \chi_{DP}(H) + \text{col}(G) \} - 1$ where $\text{col}(H)$ is the coloring number of the graph $H$. We focus on building tools for lower bound arguments for  $\chi_{DP}(G \square H)$ and use them to show the sharpness of the bound above and its various forms. Our results illustrate that the DP color function of $G$, the DP analogue of the chromatic polynomial, is essential in the study of the DP-chromatic number of the Cartesian product of graphs, including the following question that extends the sharpness problem above and the classical result on gap between list chromatic number and chromatic number: given any graph $G$ and $k \in \mathbb{N}$, what is the smallest $t$ for which $\chi_{DP}(G \square K_{k,t})= \chi_{DP}(G) + k$? 

\medskip



 \end{abstract}
 
 \section{Introduction} \label{intro}
 
 In this paper all graphs are nonempty, finite and simple unless otherwise noted.  Generally speaking we follow West~\cite{W01} for terminology and notation. We use $\mathbb{N}$ to denote the set of all natural numbers. For $k \in \mathbb{N}$, $[k]$ denotes the set $\{1,...,k\}$. For a graph $G$, $V(G)$ and $E(G)$ denote the vertex set and edge set of $G$ respectively. If $S \subseteq V(G)$, $G[S]$ denotes the subgraph of $G$ induced by $S$. For any $S_{1},S_{2} \subseteq V(G)$, $E_{G}(S_{1},S_{2})$ denotes the subset of $E(G[S_1 \cup S_2])$ with at least one end point in $S_{1}$ and at least one end point in $S_{2}$. The neighborhood of a vertex $v$ in $G$ is denoted by $N_{G}(v)$ or $N(v)$ when the graph is clear from context. The neighborhood of a set of vertices $S\subseteq V(G)$ is defined as $N(S) = \bigcup_{v \in S}N_{G}(v)$. 
 We use $K_{m,n}$ to denote the complete bipartite graphs with partite sets of size $m$ and $n$.

 \subsection{Graph coloring, list coloring, and DP-coloring}
 A \emph{proper $m$-coloring} of a graph $G$ is a function $f$ that assigns an element $f(v) \in [m]$ to each $v \in V(G)$ such that $f(v) \neq f(u)$ whenever $uv \in E(G)$. We say that $G$ is \emph{$m$-colorable} if it has a proper $m$-coloring. The \emph{chromatic number} $\chi(G)$ of $G$ is the smallest $m \in \mathbb{N}$ such that there exists a proper $m$-coloring of $G$. The \emph{coloring number} of a graph $G$, $\col(G)$, is the smallest integer $d$ for which there exists an ordering, $v_1, \ldots, v_n$, of the vertices of $G$ such that each vertex $v_i$ has at most $d-1$ neighbors among $v_1, \ldots, v_{i-1}$. Clearly, $\col(K_{k,t}) = k+1$ when $k \le t$, and $\chi(G) \leq \col(G)$ for any graph $G$.
 
 List coloring is a generalization of classical vertex coloring.  It was introduced in the 1970s independently by Vizing~\cite{V76} and Erd\H{o}s, Rubin, and Taylor~\cite{ET79}. A \emph{list assignment} of $G$ is a function $L$ on $V(G)$ that assigns a set of colors to each $v \in V(G)$. If $|L(v)| = m$ for each $v \in V(G)$, then $L$ is called an \emph{$m$-assignment} of $G$. The graph $G$ is \emph{$L$-colorable} if there exists a proper coloring $f$ of $G$ such that $f(v) \in L(v)$ for each $v \in V(G)$ (we refer to $f$ as a \emph{proper $L$-coloring} of $G$). The \emph{list-chromatic number} $\chi_{\ell}(G)$ is the smallest $m$ such that there exists a proper $L$-coloring for every $m$-assignment $L$ of $G$. It immediately follows that for any graph $G$, $\chi(G) \leq \chi_\ell(G) \leq \col(G)$. The first inequality may be strict since it is known that the gap between $\chi(G)$ and $\chi_{\ell}(G)$ can be arbitrarily large; for example, $\chi_{\ell}(K_{k,t}) = k+1$ when $t \ge k^k$ but all bipartite graphs are $2$-colorable.
 
DP-coloring is a generalization of list coloring that was introduced by Dvo\v{r}\'{a}k and Postle~\cite{DP15} in 2015. Intuitively, DP-coloring is a generalization of list coloring where each vertex in the graph still gets a list of colors but identification of which colors are different can vary from edge to edge. We now state the formal definition. A \emph{cover} of a graph $G$ is a pair $\mathcal{H} = (L,H)$, where $H$ is a graph and $L$ is function $L : V(G) \rightarrow \mathcal{P}(V(H))$ such that:

\vspace{3mm}

     \noindent(1) the set $\{L(v) : v \in V(G)\}$ forms a partition of $V(H)$ of size $|V(G)|$;\\
     (2) for every $v \in V(G)$, the graph $H[L(v)]$ is a complete graph;\\
     (3) if $E_{H}(L(u),L(v))$ is nonempty, then either $u = v$ or $uv \in E(G)$;\\
     (4) if $uv \in E(G)$, then $E_{H}(L(u),L(v))$ is a matching (the matching may be empty).
     
\vspace{3mm}     
 
We refer to the edges of $H$ connecting distinct parts of the partition $\{L(v) : v \in V(G) \}$ as \emph{cross edges}.  A cover $\mathcal{H} = (L,H)$ of $G$ is \emph{$m$-fold} if $|L(v)| = m$ for each $v \in V(G)$. An $\mathcal{H}$ -coloring of $G$ is an independent set in $H$ of size $|V(G)|$. An independent set $I \subseteq V (H)$ is an $\mathcal{H}$-coloring of $G$ if and only if $|I \cap L(v)| = 1$ for all $v \in V (G)$. If an $\mathcal{H}$-coloring of $G$ exists, then we say that $G$ \emph{admits} an $\mathcal{H}$-coloring.
 The \emph{DP-chromatic number} $\chi_{DP}(G)$ is the smallest $m \in \mathbb{N}$ such that $G$ admits an $\mathcal{H}$-coloring for every $m$-fold cover $\mathcal{H}$ of $G$.
 
 Suppose $\mathcal{H} = (L,H)$ is an $m$-fold cover of $G$ and $U \subseteq V(G)$. Let $\mathcal{H}_{U} = (L_{U},H_{U})$ where $L_{U}$ is the restriction of $L$ to $U$ and $H_{U} = H[\bigcup_{u \in U} L(u)]$. Clearly, $\mathcal{H}_{U}$ is an $m$-fold cover of $G[U]$. We call $\mathcal{H}_{U}$ the \emph{subcover of }$\mathcal{H}$\emph{ induced by }$U$. When $U = \{u\}$, we use the notation $\mathcal{H}_{u}$ instead of $\mathcal{H}_{\{u\}}$. A cover $\mathcal{H} = (L,H)$ of a graph $G$ is called a \emph{full cover} if for each $uv \in E(G)$, the matching $E_{H}(L(u),L(v))$ is perfect. We say $\mathcal{H}$ is a \emph{bad} cover of $G$ if $G$ does not admit an $\mathcal{H}$-coloring.
 
 Given an $m$-assignment, $L$, for a graph $G$, it is easy to construct an $m$-fold cover $\mathcal{H}$ of $G$ such that $G$ has an $\mathcal{H}$-coloring if and only if $G$ has a proper $L$-coloring (see~\cite{BK17}).  It follows that $\chi_\ell(G) \leq \chi_{DP}(G)$.  This inequality may be strict since it is easy to prove that $\chi_{DP}(C_n) = 3$ whenever $n \geq 3$, but the list chromatic number of any even cycle is $2$ (see~\cite{BK17} and~\cite{ET79}).
 
\subsection{Chromatic Polynomial, List Color Function, and DP Color Function} \label{DPCF}
 
For $m \in \N$, the \emph{chromatic polynomial} of a graph $G$, $P(G,m)$, is equal to the number of proper $m$-colorings of $G$. It can be shown that $P(G,m)$ is a polynomial in $m$ of degree $|V(G)|$ (see~\cite{B12}).

In 1990 the notion of chromatic polynomial was extended to list coloring as follows~\cite{AS90}. If $L$ is a list assignment for $G$, we use $P(G,L)$ to denote the number of proper $L$-colorings of $G$. The \emph{list color function} $P_\ell(G,m)$ is the minimum value of $P(G,L)$ where the minimum is taken over all possible $m$-assignments $L$ for $G$.  Since an $m$-assignment could assign the same $m$ colors to every vertex in a graph, it is clear that $P_\ell(G,m) \leq P(G,m)$ for each $m \in \N$. It is known that for each $m \in \N$, $P(G,m)=P_{\ell}(G,m)$ when $G$ is a cycle or chordal\footnote{A chordal graph is a graph in which all cycles of length four or more contain a chord.} (see~\cite{KN16} and~\cite{AS90}). But for some graphs, the list color function can differ significantly from the chromatic polynomial for small values of $m$.  One reason for this is that a graph can have a list chromatic number that is much higher than its chromatic number. On the other hand, Wang, Qian, and Yan~\cite{WQ17} (improving upon results in~\cite{D92} and~\cite{T09}) showed that for a connected graph $G$ with $t$ edges, $P_{\ell}(G,m)=P(G,m)$ whenever $m > {(t-1)}/{\ln(1+ \sqrt{2})}$.

Recently, the notion of chromatic polynomial was extended to DP-coloring~\cite{KM20}.  Suppose $\mathcal{H} = (L,H)$ is a cover of graph $G$.  Let $P_{DP}(G, \mathcal{H})$ be the number of $\mathcal{H}$-colorings of $G$.  Then, the \emph{DP color function of $G$}, denoted $P_{DP}(G,m)$, is the minimum value of $P_{DP}(G, \mathcal{H})$ where the minimum is taken over all possible $m$-fold covers $\mathcal{H}$ of $G$.  It is easy to see that for any graph $G$ and $m \in \N$, $P_{DP}(G, m) \leq P_\ell(G,m) \leq P(G,m)$.  It is also fairly straightforward to prove that for each $n \geq 3$ and $m \in \N$, $P_{DP}(C_n,m) = P(C_n,m)$ when $n$ is odd and $P_{DP}(C_n,m) = (m-1)^n - 1$ when $n$ is even and $m \geq 2$ (see~\cite{KM20} or~\cite{M21}).

As the DP color function of an even cycle demonstrates, unlike the list color function, the DP color function need not be equal to the chromatic polynomial even for sufficiently large values of $m$.  In fact, Dong and Yang~\cite{DY21} recently (extending results of~\cite{KM20}) showed that if $G$ contains an edge $e$ such that the length of a shortest cycle containing $e$ in $G$ is even, then there exists $N \in \mathbb{N}$ such that $P_{DP}(G,m) < P(G,m)$ whenever $m \geq N$. In general, it was shown in~\cite{MT20} that  for every $n$-vertex graph $G$, $P(G, m)-P_{DP}(G, m) = O(m^{n-3})$ as $m \rightarrow \infty$; it follows that for any graph $G$ whose $P_{DP}(G,m)$ is a polynomial in $m$ for large enough $m$, the polynomial will have the same three terms of highest degree as $P(G, m)$.

 \subsection{List Coloring Cartesian Products of Graphs}

The \emph{Cartesian product} of graphs $G$ and $H$, denoted $G \square H$, is the graph with vertex set $V(G) \times V(H)$ and edges created so that $(u,v)$ is adjacent to $(u',v')$ if and only if either $u=u'$ and $vv' \in E(H)$ or $v=v'$ and $uu' \in E(G)$. Every connected graph has a unique factorization under this graph product (\cite{S60}), and this factorization can be found in linear time and space (\cite{IP07}).

It is well-known that $\chi(G \square H) = \max \{\chi(G), \chi(H) \}$.  For the list-chromatic number, Borowiecki, Jendrol, Kr{\'a}l, and Mi{\v s}kuf~\cite{BJ06} showed the following in 2006.

\begin{thm}[\cite{BJ06}] \label{thm: Borow1}
For any graphs $G$ and $H$, $\chi_\ell(G \square H) \leq \min \{\chi_\ell(G) + \col(H), \col(G) + \chi_\ell(H) \} - 1.$
\end{thm}

  For any graph $G$, it is easy to see that Theorem~\ref{thm: Borow1} implies $\chi_\ell(G \square K_{k,t}) \leq \chi_\ell(G) + k$.  The following result by Kaul and Mudrock~\cite{KM18} demonstrates the sharpness of Theorem~\ref{thm: Borow1}.

\begin{thm} [\cite{KM18}] \label{thm: KM18}
Let $G$ be any graph.  Then, $\chi_\ell(G \square K_{k,t}) = \chi_\ell(G) + k$ whenever $t \geq (P_\ell(G, \chi_\ell(G) + k - 1))^k$.
\end{thm}

This was an improvement on Borowiecki, Jendrol, Kr{\'a}l, and Mi{\v s}kuf's earlier result, $\chi_\ell(G \square K_{k,t}) = \chi_\ell(G) + k$ whenever $t \geq (\chi_\ell(G) + k - 1)^{k|V(G)|}$ which is a generalization of the classical result on the list chromatic number of a complete bipartite graph; $\chi_\ell(K_{k,t}) = k+1$ if and only if $t \geq k^k$ (see~\cite{BJ06}). A motivation for this paper was to extend Theorems~\ref{thm: Borow1} and~\ref{thm: KM18} to DP-coloring.
 
 \subsection{Outline of Results and Open Questions}
 
 In this section we present an outline of the paper while also mentioning some open questions. We begin Section~\ref{two} by proving the DP-analogue of Theorem~\ref{thm: Borow1}.
 
 \begin{thm} \label{thm: cartprod}
    For any graphs $G$ and $H$, $\chi_{DP}(G\square H) \leq \min\{\chi_{DP}(G) + \col(H),\chi_{DP}(H) + \col(G)\} - 1$.
\end{thm}

We conclude Section~\ref{two} by showing that the bound in Theorem~\ref{thm: cartprod} is sharp. 

\begin{thm} \label{thm: cartprodcompbipartite}
	For any graph $G$, $\chi_{DP}(G\square K_{k,t}) = \chi_{DP}(G) + k$ whenever $t \geq (P_{DP}(G,\chi_{DP}(G)+k-1))^{k}$.
\end{thm}

To prove this, we define the notion of volatile coloring which gives a characterization of the bad covers of the Cartesian product of an arbitrary graph with a complete bipartite graph. Volatile coloring, and conditions derived from it, may be of independent interest in the study of the DP-chromatic number of other Cartesian products. In the remainder of the paper we will show how volatile colorings are a useful tool for lower bound arguments for the DP-chromatic number.

Building upon Theorem~\ref{thm: cartprodcompbipartite}, in the rest of the paper, we will show evidence that the DP color function is a useful tool in the study of the DP-chromatic number of the Cartesian product of graphs. Considering the sharpness of Theorem~\ref{thm: cartprodcompbipartite}, also inspires us to consider a more general question which is the focus of Sections~\ref{three} and~\ref{four}.

\begin{ques} \label{ques: 2}
Given a graph $G$ and $k \in \N$, let $f(G,k)$ be the function satisfying: $\chi_{DP}(G\square K_{k,t}) \allowbreak= \chi_{DP}(G)+k$ if and only if $t \geq f(G,k)$. What is $f(G,k)$?\footnote{Since $\chi_{DP}(G\square K_{k,0}) = \chi_{DP}(G) < \chi_{DP}(G) + k$, we have $f(G,k)\geq 1$. Moreover by Theorem~\ref{thm: cartprodcompbipartite}, $f(G,k) \leq (P_{DP}(G,\chi_{DP}(G)+k-1))^{k}$. Hence $f(G,k)$ exists for every $G$ and $k \in \N$.}
\end{ques}

Note that $f(G,k)$ is the smallest $t$ such that $\chi_{DP}(G\square K_{k,t}) = \chi_{DP}(G) + k$ (we use this notation for the remainder of the paper).

We begin Section~\ref{three} with definitions of canonical labeling and twisted-canonical labeling of covers which give a characterization of the bad covers of odd and even cycles respectively (see~\cite{KO19} and~\cite{LWW22} for more general definitions and results related to bad covers of graphs which imply our characterizations for cycles\footnote{We thank the referees for bringing these papers to our attention.}). These notions of labelings are of independent interest in the study of DP coloring (see e.g., ~\cite{CL1,CL2,MT20}). Using these tools along with volatile coloring, we end Section~\ref{three} by showing that Theorem~\ref{thm: cartprodcompbipartite} is sharp when $G$ is an even cycle and $k=1$; that is, for any $m \in \N$, $f(C_{2m+2},1) = P_{DP}(C_{2m+2},3) = 2^{2m+2}-1$.  This is the only sharpness example that we have found which makes us think that Theorem~\ref{thm: cartprodcompbipartite} might not be sharp very often. This motivates us to ask the following question.

\begin{ques} \label{ques: 3}
Does there exist a graph $G$ such that $f(G,k) = (P_{DP}(G,\chi_{DP}(G)+k-1))^{k}$ for every $k \in \N$?
\end{ques}

Questions~\ref{ques: 2} and~\ref{ques: 3} seem hard to answer completely since there is no $G$ for which we know $f(G,k)$ for every $k \in \N$. Even in the case when $G = K_{1}$, the form of this question which has been studied before,  $f(K_{1},k)$ is not known exactly. Mudrock~\cite{M18} showed that $k^{k}/k! < f(K_{1},k) \leq 1+(k^{k}/k!)(\log(k!)+1)$. This also demonstrates that Theorem~\ref{thm: cartprodcompbipartite} is not always sharp. Consider the case when $k=3$, $\chi_{DP}(K_{1}\square K_{3,t}) = 4$ whenever $t \geq 27$ by Theorem~\ref{thm: cartprodcompbipartite} (note that $\chi_{DP}(K_{1}) = 1$ and $P_{DP}(K_{1},k) = k$); whereas by Mudrock's result, $\chi_{DP}(K_{1}\square K_{3,t}) = 4$ whenever $t \geq 10$.

In Section~\ref{four}, we make progress towards Question~\ref{ques: 2} by improving the bound in Theorem~\ref{thm: cartprodcompbipartite} when $G$ is a cycle. Our arguments also illustrate the subtle difference in handling even cycles versus odd cycles in DP coloring, as also evident in results on the DP color function (see Section~\ref{DPCF}). 

By extending the proof ideas of Theorem~\ref{thm: cartprodcompbipartite} and using the tools from Section~\ref{three}, we construct random covers using a combination of random matchings defined using an equivalence relation on an appropriate set of colorings, and matchings defined using canonical and twisted-canonical labelings. We show that there exists an appropriate bad cover by counting the expected number of volatile colorings.

\begin{thm} \label{thm: oddrandom}
        Given $k \in \mathbb{N}$, let $c_k = \left \lceil{\frac{k\ln(k+2)}{\ln{2}+(k-1)\ln(k+2)-\ln(2(k+2)^{k-1}-(k+1)!)}} \right \rceil$ if $k \geq 2$ and $c_1=1$. Then,\\ $\chi_{DP}(C_{2m+1}\square \allowbreak K_{k,t}) = k+3$ whenever $t\geq c_k\left(\frac{P_{DP}(C_{2m+1},k + 2)}{k+2}\right)^k = c_k\left(\frac{(k+1)^{2m+1} - (k+1)}{k+2}\right)^k$.
    \end{thm}
    
For example, the theorem requires $c_k = 1, 3, 8$ when $k = 1, 2, 3$ respectively. It is easy to see that $c_k < \left(\frac{2k\ln(k+2)}{(k+1)!}\right)(k+2)^{k}$, and hence $\left(\frac{2k\ln(k+2)}{(k+1)!}\right) (P_{DP}(C_{2m+1},k + 2))^k$ suffices as a lower bound on $t$, a strong improvement on Theorem~\ref{thm: cartprodcompbipartite} when $G$ is an odd cycle. 

Note that Theorem~\ref{thm: oddrandom} implies that $f(C_{2m+1},k) \leq c_k\left({P_{DP}(C_{2m+1},k + 2)}/{(k+2)}\right)^k$. Next we show that Theorem~\ref{thm: oddrandom} is sharp when $k=1$; specifically, $f(C_{2m+1},1) = {P_{DP}(C_{2m + 1},3)}/{3} = {(2^{2m+1}-2)}/{3}$. For list coloring, it is shown in~\cite{KM18} that  $g(C_{2m+1},1) =$ $P_{\ell}(C_{2m+1},3)$ $=(2^{2m+1}-2)$ where $g(C_{2m+1},1)$ is the list coloring analogue~\footnote{$g(G,k)$ is the function satisfying: $\chi_{\ell}(G \square K_{k,t}) =\chi_{\ell}(G) + k$ if and only if $t \geq g(G,k)$.} of $f(C_{2m+1},1)$. 

We conclude Section~\ref{four} by proving the even cycle analogue of Theorem~\ref{thm: oddrandom}.

\begin{thm} \label{thm: evenrandom}
        Given $k \in \mathbb{N}$, let $c_k = \left \lceil{\frac{k\ln(k+2)}{k\ln(k+2)-\ln((k+2)^{k}-\floor{(k+2)/2}k!)}}\right \rceil$. Then,\\ $\chi_{DP}(C_{2m+2}\square K_{k,t}) \allowbreak= k+3$ whenever $t\geq c_k\left(\frac{P_{DP}(C_{2m+2},k + 2)}{k+2}\right)^k = c_k\left(\frac{(k+1)^{2m+2}-1}{k+2}\right)^k$.
\end{thm}
    
For example, the theorem requires $c_k = 3, 10, 48$ when $k = 1, 2, 3$ respectively. It is easy to see that $c_k < \left(\frac{2\ln(k+2)}{\floor{(k+2)/2}(k-1)!}\right)(k+2)^k$, and hence $\left(\frac{2\ln(k+2)}{\floor{(k+2)/2}(k-1)!}\right) (P_{DP}(C_{2m+2},k + 2))^k$ suffices as a lower bound on $t$, a strong improvement on Theorem~\ref{thm: cartprodcompbipartite} when $G$ is an even cycle. 

Note that Theorem~\ref{thm: evenrandom} implies that $f(C_{2m+2},k) \leq c_k\left({P_{DP}(C_{2m+2},k + 2)}/{(k+2)}\right)^k$. The sharpness of Theorem~\ref{thm: evenrandom} follows from the earlier result: $f(C_{2m+2},1) = P_{DP}(C_{2m+2},3) = 2^{2m+2}-1$.

\section{DP-coloring Cartesian products of Graphs} \label{two}

We now prove Theorem~\ref{thm: cartprod} using a straightforward generalization of Borowiecki et al.'s argument~\cite{BJ06} for Theorem~\ref{thm: Borow1}.

\begin{customthm} {\bf \ref{thm: cartprod}}
    For any graphs $G$ and $H$, $\chi_{DP}(G\square H) \leq \min\{\chi_{DP}(G) + \col(H),\chi_{DP}(H) + \col(G)\} - 1$.
\end{customthm}

\begin{proof}
Since the Cartesian product of graphs is commutative, we assume without loss of generality that $\chi_{DP}(G) + \col(H)\leq \chi_{DP}(H) + \col(G)$. We let $m = \chi_{DP}(G)$, $k = \col(H)$, and $d = \chi_{DP}(G) + \col(H) - 1$.  Suppose $V(H) = \{v_{i} : i \in [n]\}$. Let $\mathcal{H} = (L, M)$ be a $d$-fold cover of $G\square H$. To prove the bound, we will show the existence of an independent set of size $n|V(G)|$ in $M$. We proceed by induction on $n$. If $n = 1$, then $G\square H \cong G$ and clearly $k = 1$. So there exists an $\mathcal{H}$-coloring of $G\square H$ since $d = \chi_{DP}(G)$.

Now suppose $n > 1$. Since $k = \col(H)$, there exists an ordering of the vertices of $H$ such that each vertex has at most $k-1$ neighbors preceding it in that ordering. Suppose $v_1, \ldots, v_n$ is one such ordering. Let $H' = H[\{v_1, \hdots, v_{n - 1}\}]$. We now construct a $d$-fold cover of $G\square H'$. First, we define the function $L'$ on $V(G\square H')$ so that for each $(u,v)\in V(G\square H')$, $L'(u,v) = L(u,v)$. Next, let $S = \bigcup_{(u,v)\in V(G\square H')}L(u,v)$ and let $M' = M[S]$. Notice that $\mathcal{H}' = (L',M')$ is a $d$-fold cover of $G\square H'$.
By the inductive hypothesis, $\chi_{DP}(G\square H')\leq \chi_{DP}(G) + \col(H') - 1$ and since $\col(H')\leq\col(H)$, we have $\chi_{DP}(G\square H')\leq d$. This implies that there exists an independent set $I'$ in $M'$ of size $(n - 1)|V(G)|$. 
We now extend $I'$ to an $\mathcal{H}$-coloring of $G \square H$. For each $u\in V(G)$, let $F_u=L(u,v_n)\cap N_M(I')$. Note that for each $i\in [n - 1]$, $|I' \cap L(u,v_i)| = 1$ and in $H$, $v_n$ has at most $(k - 1)$ neighbors in $\{v_1, \hdots, v_{n - 1}\}$. So for each $u\in V(G)$, $|F_u| \leq (k - 1)$ and thus $|L(u,v_n) - F_u| \geq m$. For each $u\in V(G)$, let $A_u$ be an $m$-element subset  of $L(u,v_n) - F_u$. We define a function, $\hat{L}$ on $V(G)\times \{v_n\}$ such that for each $u\in V(G)$, $\hat{L}(u,v_n) = A_u$. Let $A = \bigcup_{u\in V(G)} A_u$ and let $\hat{M} = M[A]$. Clearly, $\hat{\mathcal{H}} = (\hat{L},\hat{M})$ is an $m$-fold cover of $G\square H[\{v_n\}]$. Note that $G\square H[\{v_n\}]\cong G$ and since $m = \chi_{DP}(G)$, there exists an independent set $\hat{I}$ of size $|V(G)|$ in $\hat{M}$. Since no vertex in $\hat{I}$ is adjacent to any vertex in $I'$, $I = \hat{I}\cup I'$ is an independent set of size $n|V(G)|$ in $M$ as needed. 
\end{proof}

Next we wish to show that Theorem~\ref{thm: cartprod} is sharp.  To do this, we introduce the notion of volatile coloring. As we will see throughout the paper, volatile coloring is an important tool in the process of constructing bad covers of the Cartesian product of an arbitrary graph and a complete bipartite graph. In particular, it gives a necessary and sufficient condition for a cover to be bad. Though we define the notion of volatile coloring specifically for the Cartesian products with a complete bipartite factor, this definition can easily be generalized for the Cartesian products of any two graphs.

\subsection{Volatile Coloring}\label{sec:volatile}

Let $G$ be a graph with $V(G) = \{v_{i} : i \in [n]\}$ and $\chi_{DP}(G) = m$. Let $K$ be a copy of the complete bipartite graph $K_{k,t}$ with partite sets, $X = \{x_{j} : j \in [k]\}$ and $Y = \{y_{q} : q \in [t]\}$. Let $M = G \square K$ and $M_{X} = M[\{(v_{i},x_{j}) : i \in [n], j \in [k]\}]$, and for each $q \in [t]$, let $M_{y_{q}} = M[\{(v_{i},y_{q}) : i \in [n]\}]$. Let $\mathcal{H} = (L,H)$ be an $(m+k-1)$-fold cover of $M$.

Let $\mathcal{H}_{X} = (L_{X}, H_{X})$ denote the subcover of $\mathcal{H}$ induced by $V(G)\times X$. Recall that $L_{X}$ is the restriction of $L$ to $V(G)\times X$ and $H_{X} = H[\bigcup_{v \in L(V(G)\times X)} L(v)]$. Similarly, for each $q \in [t]$, let $\mathcal{H}_{y_{q}} = (L_{y_{q}}, H_{y_{q}})$ denote the subcover of $\mathcal{H}$ induced by $V(G)\times \{y_{q}\}$. 

Suppose $I$ is an $\mathcal{H}_{X}$-coloring of $M_{X}$. For each $q \in [t]$, let $D_{q} = \{u \in V(H_{y_{q}}) : N_{H}(u)\cap I = \emptyset\}$ and $H'_{y_{q}} = H[D_{q}]$. For each $u \in V(M_{y_{q}})$, we define $L'_{y_{q}}(u) = L_{y_{q}}(u)\cap D_{q}$. Let $\mathcal{H}'_{y_{q}} = (L'_{y_{q}}, H'_{y_{q}})$.
We say $I$ is \emph{volatile} for $M_{y_{q}}$, if $\mathcal{H}'_{y_{q}}$ is a bad cover of $M_{y_{q}}$.

Assuming the same setup as in the definition of volatile coloring above, we now give a necessary and sufficient condition for a cover to be bad.

\begin{lem} \label{lem: nohcoloring}
For each $\mathcal{H}_{X}$-coloring of $M_{X}$, $I_{X}$, there exists a $q \in [t]$ such that $I_{X}$ is volatile for $M_{y_{q}}$ if and only if $\mathcal{H}$ is a bad cover of $M$.

\end{lem}
\begin{proof}
Suppose there exists an $\mathcal{H}_{X}$-coloring of $M_{X}$, $I_{X}$, such that for each $q \in  [t]$, $I_{X}$ is not volatile for $M_{y_{q}}$.  Since for each $q \in  [t]$, $I_{X}$ is not volatile for $M_{y_{q}}$, there exists an $\mathcal{H}'_{y_{q}}$-coloring, $I_{q}$, of $M_{y_{q}}$. Let $I_{Y} = \bigcup_{q \in [t]} I_{q}$. Clearly, $I_{X}\cup I_{Y}$ is an $\mathcal{H}$-coloring of $M$ as desired.

    Conversely suppose $M$ admits an $\mathcal{H}$-coloring, $I$. Let $I_{X} = I \cap V(H_{X})$. Clearly, $I_{X}$ is an $\mathcal{H}_{X}$-coloring of $M_{X}$. Suppose there exists an $r\in[t]$ such that $I_X$ is volatile for $M_{y_{r}}$.  Notice that $I \cap V(H_{y_{r}}) \subseteq D_{r}$. This means $I \cap V(H_{y_{r}})$ is an $\mathcal{H}'_{y_{r}}$-coloring of $M_{y_{r}}$. This is a contradiction since $I_{X}$ is volatile for $M_{y_{r}}$.
\end{proof}

By a straightforward application of Lemma~\ref{lem: nohcoloring}, next we give a sufficient condition for $M$ to admit an $\mathcal{H}$-coloring which will often be used in the remaining sections.

\begin{cor} \label{cor: hcoloring}
Let $c$ be the number of $\mathcal{H}_{X}$-colorings of $M_{X}$. Suppose for each $q \in [t]$, the number of volatile $\mathcal{H}_{X}$-colorings for $M_{y_{q}}$ is at most $z$. If $c > zt$ then $M$ admits an $\mathcal{H}$-coloring.
\end{cor}

Using Lemma~\ref{lem: nohcoloring}, we are now ready to prove the sharpness of Theorem~\ref{thm: cartprod}.

\begin{customthm} {\bf \ref{thm: cartprodcompbipartite}}
    For any graph $G$, $\chi_{DP}(G\square K_{k,t}) = \chi_{DP}(G) + k$ whenever $t \geq (P_{DP}(G,\chi_{DP}(G)+k-1))^{k}$.
\end{customthm}

\begin{proof}
Let $\chi_{DP}(G) = m$ and $K$ be a complete bipartite graph with partite sets $X = \{x_{j} : j \in [k]\}$ and $Y = \{y_{q} : q \in [t]\}$. We construct a bad $(m+k-1)$-fold cover $\mathcal{H} = (L,H)$ of $G\square K$. For each $(u,v) \in V(G\square K)$, we let $L(u,v) = \{(u,v,i) : i \in [m+k-1]\}$. We begin constructing the graph $H$ by defining the vertex set, $V(H) = \bigcup_{(u,v) \in V(G \square K)} L(u,v)$. We then create edges in $H$ such that for each $(u,v) \in V(G \square K)$, $H[L(u,v)]$ is a complete graph on $(m+k-1)$ vertices. Next, we construct the matchings among these cliques such that $G\square K$ does not admit an $\mathcal{H}$-coloring.

Let $P_{DP}(G,m+k-1) = d$; suppose $\mathcal{H}_{G} = (L_{G},H_{G})$ is an $(m+k-1)$-fold cover of $G$ such that $P_{DP}(G,\mathcal{H}_{G}) = d$. For each $u \in V(G)$, suppose $L_{G}(u) = \{(u,l) : l \in [m+k-1]\}$. We denote the collection of all $\mathcal{H}_{G}$-colorings of $G$ by $\mathcal{I}_{G} = \{I_{{G}_{i}}: i \in [d]\}.$

For each $x \in X$, we create edges so that $(u,x,l)(v,x,j) \in E(H)$ for all distinct $u,v \in V(G)$ whenever $(u,l)(v,j) \in E(H_{G})$. Notice that for each $x$, $H[V(G) \times \{x\} \times [m+k-1]]$ is isomorphic to $H_{G}$ and for each $j\in [k]$ the isomorphism is defined as: $g^{(j)} : V(H_{G}) \rightarrow V(G) \times \{x_{j}\} \times [m+k-1] \text{ such that }g^{(j)}(u,l) = (u,x_{j},l)$.  For each $I_{{G}_{i}} \in \mathcal{I}_{G}$, we let $I_{ij} = \{g^{(j)}(a) :  a \in I_{{G}_{i}}\}$. For each $j \in [k]$, let $\mathcal{I}_{{G},{j}}=\{I_{ij} : i \in [d]\}$

Notice that by construction all possible independent sets containing one vertex from each clique, $L(u,x_{j})$ whenever $(u,x_{j}) \in V(G) \times X$ can be formed by picking an independent set from $\mathcal{I}_{{G},{j}}$ for each $j\in [k]$. Clearly there are $d$ choices for each $j \in [k]$ giving us a total of $d^k$ possible independent sets. We denote the collection of these independent sets by $\mathcal{I} = \{I_{i} : i \in [d^k]\}$.
 
For each $I_{i} \in \mathcal{I}$ and $u \in V(G)$, we define, $I_{{i}_{u}} = I_{i} \cap (\bigcup_{x \in X} L(u,x))$. Clearly, $|I_{{i}_{u}}| = k$ for each $u \in V(G)$. We name the vertices in $I_{{i}_{u}}$ so that $I_{{i}_{u}} = \{(u,x_{j},z_{i,u,j}) : j \in [k]\}$. Note that $z_{i,u,j}\in [m + k - 1]$. For each $I_i\in\mathcal{I}$, $u \in V(G)$, and $j\in [k]$, we create edges so that$(u,x_{j},z_{i,u,j})(u,y_{i},j) \in E(H)$. Notice that this can be done since $t \geq d^{k}$.

We now create the remaining edges in $H$. Let $\mathcal{H}'_{G} = (L'_{G},H'_{G})$ be a bad $(m-1)$-fold cover of $G$. For each $u \in V(G)$, suppose $L'_{G}(u) = \{(u,l) : l \in [m-1]\}$. For each $q \in [t]$, we create edges so that $(u,y_{q},k+l)(v,y_{q},k+j) \in E(H)$ for all distinct $u,v \in V(G)$ whenever $(u,l)(v,j) \in E(H'_{G})$. This completes the construction of $H$.

For each $L(u,y_{i})$ where $(u,y_{i}) \in V(G) \times Y$, let $W_{u,y_{i}} = \{(u,y_{i},j) : j \in [k]\}$. Recall that we created edges from vertices in $W_{u,y_{i}}$ to $I_{{i}_{u}}$ for each $i\in [d^k]$. Now, for each $i \in [t]$, we define $H_{G}^{(i)} = H[\bigcup_{u \in V(G)} (L(u,y_{i}) - W_{u,y_{i}})]$. Clearly, each $H_{G}^{(i)}$ is isomorphic to $H'_{G}$ and the isomorphism is defined as: $f^{(i)} : V(H'_{G}) \rightarrow V(H_{G}^{(i)})\text{ such that }f^{(i)}(u,l) = (u,y_{i},k+l)$.  For each $q \in [t]$, let $M_{y_{q}} = M[\{(v,y_{q}) : v \in V(G)\}]$. By construction, for each $i \in [d^k]$,  $I_{i}$ is volatile for $M_{y_{i}}$ and thus by Lemma~\ref{lem: nohcoloring}, $\mathcal{H}$ is a bad cover of $G\square K$.
\end{proof}

We are now ready to turn our attention to making progress on Question~\ref{ques: 2}.

\section{Characterization of Bad Covers of Cycles} \label{three}
In this Section, we define the notions of canonical labeling and twisted-canonical labeling which characterize the bad covers of odd and even cycles respectively(see~\cite{KO19} and~\cite{LWW22} for more generalized results). Canonical labeling, in particular, is of independent interest and has been used in many recent papers on DP-coloring (see e.g.,~\cite{CL1},~\cite{CL2}, and~\cite{MT20}). As an application of these together with volatile coloring, we are able to show the sharpness of Theorem~\ref{thm: cartprodcompbipartite}. This also makes progress towards Question~\ref{ques: 2}. 

\subsection{Canonical Labeling}

Suppose $G$ is a graph and $\mathcal{H} = (L,H)$ is a $k$-fold cover of $G$. We say $\mathcal{H}$ has a \emph{canonical labeling} if for each $v \in V(G)$, it is possible to let $L(v) = \{(v,j) : j \in [k]\}$ so that whenever $uv \in E(G)$, $(u,j)$ and $(v,j)$ are adjacent in $H$ for each $j \in [k]$.

\begin{figure}[h]
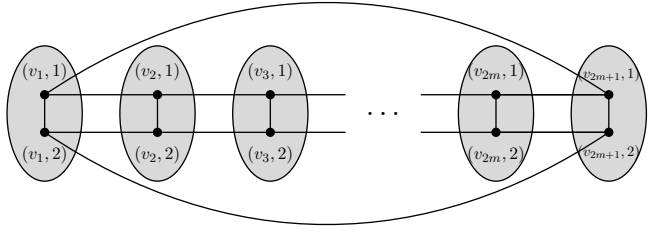

\centering
\begin{graph}

\coordinate (v11) at (-1.5,0);
        \coordinate (v12) at (-1.5,0.5);
        \coordinate (v21) at (0,0);
        \coordinate (v22) at (0,0.5);
        \coordinate (v31) at (1.5,0);
        \coordinate (v32) at (1.5,0.5);
        \coordinate (dot) at (3,0.25);
        \coordinate (v41) at (4.5,0);
        \coordinate (v42) at (4.5,0.5);
        \coordinate (v51) at (6,0);
        \coordinate (v52) at (6,0.5);
        
        \draw[fill=gray!30] (-1.5,0.25) ellipse (5mm and 9mm);
        \draw[fill=gray!30] (0,0.25) ellipse (5mm and 9mm);
        \draw[fill=gray!30] (1.5,0.25) ellipse (5mm and 9mm);
        
        \draw[fill=gray!30] (4.5,0.25) ellipse (5mm and 9mm);
        \draw[fill=gray!30] (6,0.25) ellipse (5mm and 9mm);
        
        \node at (dot) {$\hdots$};
        \draw[fill=black] (v11) circle[radius=1.5pt] node[below=3pt,scale=0.6] {$(v_1,2)$};
        \draw[fill=black] (v12) circle[radius=1.5pt] node[above=3pt,scale=0.6] {$(v_1,1)$};
        \draw[fill=black] (v21) circle[radius=1.5pt] node[below=3pt,scale=0.6] {$(v_2,2)$};
        \draw[fill=black] (v22) circle[radius=1.5pt] node[above=3pt,scale=0.6] {$(v_2,1)$};
        \draw[fill=black] (v31) circle[radius=1.5pt] node[below=3pt,scale=0.6] {$(v_3,2)$};
        \draw[fill=black] (v32) circle[radius=1.5pt] node[above=3pt,scale=0.6] {$(v_3,1)$};
        \draw[fill=black] (v42) circle[radius=1.5pt] node[above=3pt,scale=0.6] {$(v_{2m},1)$};
        \draw[fill=black] (v41) circle[radius=1.5pt] node[below=3pt,scale=0.6] {$(v_{2m},2)$};
        \draw[fill=black] (v52) circle[radius=1.5pt] node[above=3pt,scale=0.5] {$(v_{2m+1},1)$};
        \draw[fill=black] (v51) circle[radius=1.5pt] node[below=3pt,scale=0.5] {$(v_{2m+1},2)$};
        \draw (v11) -- (v12); \draw (v21) -- (v22); \draw (v31) -- (v32); \draw (v41) -- (v42); \draw (v51) -- (v52);
        
        \draw (v11) -- (v21); \draw (v21) -- (v31);
        \draw (v12) -- (v22); \draw (v22) -- (v32);
        \draw (v32) -- (2.5,0.5);
        \draw (v31) -- (2.5,0);
        \draw (3.5,0.5) -- (v52);
        \draw (3.5,0) -- (v51);
        \draw (v41) -- (v51);
        \draw (v42) -- (v52);
        \path[line width=0.5] (v11) edge[bend right=34] (v51);
        \path[line width=0.5] (v12) edge[bend right=-34] (v52);

\end{graph}
\caption{A $2$-fold cover of an odd cycle with a canonical labeling}
\label{figure:CL}
\end{figure}

Note that if $G$ is a graph and $\mathcal{H}$ is an $m$-fold cover of $G$ with a canonical labeling, then $G$ has a proper $m$-coloring if and only if $G$ admits an $\mathcal{H}$-coloring. We next restate a result of Kaul and Mudrock~\cite{KM20} using the concept of canonical labeling.

\begin{pro}[\cite{KM20}] \label{prop: treecl}
Let $T$ be a tree and $\mathcal{H} = (L,H)$ be a full $m$-fold cover of $T$. Then, $\mathcal{H}$ has a canonical labeling.
\end{pro}
Using the notion of canonical labeling, we can characterize the bad 2-fold covers of odd cycles. See~\cite{KO19} for a more general result about bad covers non-DP-degree colorable graphs, which also implies the following result. Note that in~\cite{KO19}, a cover with canonical labeling is called a \emph{ladder}. 
\begin{lem} \label{lem: oddchar}
Suppose that $G$ is an odd cycle and $\mathcal{H} = (L,H)$ is a cover of $G$ where $|L(v)| \geq 2$ for each $v \in V(G)$. Then $G$ does not admit an $\mathcal{H}$-coloring if and only if $\mathcal{H}$ is a $2$-fold cover with a canonical labeling.
\end{lem}

\begin{proof}
   Let $G = C_{2m+1}$ where $m\in\mathbb{N}$. Suppose $\mathcal{H}$ is a $2$-fold cover with a canonical labeling. The fact that $G$ does not admit an $\mathcal{H}$-coloring immediately follows from the fact that $G$ is not 2-colorable.

Conversely, suppose first that $\mathcal{H}$ is a 2-fold cover without a canonical labeling. We can assume $\mathcal{H}$ is a full cover. Otherwise, we could define $H'$ such that $V(H') = V(H)$ and $E(H)\subseteq E(H')$ but $E_{H'}(L(u),L(v))$ is a perfect matching for each $uv\in E(G)$, and note that if $I$ is an independent set in $H'$, then it will be an independent set in $H$. For any edge $uv\in E(G)$, the cover $(L, H - E(L(u),L(v)))$ of $G - uv$ has a canonical labeling by Proposition~\ref{prop: treecl}. Hence, we can name the vertices of $H$ such that for each $xy\in E(G) - \{uv\}$, the edges $(x,1)(y,1)$ and $(x,2)(y,2)$ are in $E(H)$. Since we know that $\mathcal{H}$ does not have a canonical labeling, it must then be the case that $(u,1)(v,2)$ and $(u,2)(v,1)$ are in $E(H)$. So, we can construct an independent set $I$ in $H$ as follows. Suppose the vertices of $G$ are ordered cyclically as $x_1,x_2,\hdots,x_{2m+1}$ where $x_1 = u$ and $x_{2m+1} = v$. Then consider the set $I = \{(x_i,1) : i\text{ is odd},1\leq i\leq 2m + 1\}\cup \{(x_i,2) : i\text{ is even},2\leq i\leq 2m\}$. Clearly $|I| = 2m + 1$, and we claim that $I$ is an independent set. First, note that there is no edge between the vertices selected from $L(x_i)$ and $L(x_{i+1})$ for $i = 1,\ldots,2m$. Finally, $(u,1),(v,1)\in I$, and as argued above, $(u,1)(v,1)\notin E(H)$. So $I$ is an independent set of size $2m + 1$ in $H$, and hence $G$ admits an $\mathcal{H}$-coloring.
    
    Now suppose $\mathcal{H}$ is not 2-fold.  Then we know that there must exist some $u\in V(G)$ such that $|L(u)|\geq 3$. Suppose the vertices of $G$ are ordered cyclically as $x_1,\hdots,x_{2m+1}$ where $x_{2m+1} = u$. Then we can construct an $\mathcal{H}$-coloring $I$ greedily by selecting for each $x_i$ some vertex $v_i$ in $L(x_i)$ that is not adjacent in $H$ to any vertex in $\{v_1, \ldots, v_{i-1} \}$. Since each $x_i$ for $i\in[2m]$ has at most one neighbor preceding it in the ordering there will always exist such a vertex $v_i$ for each $i\in[2m]$. Finally, since $x_{2m+1}$ has two neighbors preceding it but $|L(x_{2m+1})|\geq 3$, there must exist a $v_{2m+1} \in L(x_{2m+1})$ which can be added to $I$. Thus, $G$ admits an $\mathcal{H}$-coloring.
\end{proof}

\subsection{Twisted-Canonical Labeling}

Suppose $G$ is a graph and $\mathcal{H} = (L,H)$ is a $k$-fold cover of $G$. We say $\mathcal{H}$ has a \emph{twisted-canonical labeling} if $\mathcal{H}$ is full and it is possible to let $L(v) = \{(v,j) : j \in [k]\}$ for each $v \in V(G)$ and choose two adjacent vertices, $u$ and $v$  in $G$ so that whenever $xy \in E(G)-\{uv\}$, $(x,j)$ and $(y,j)$ are adjacent in $H$ for each $j \in [k]$ and there exists $l \in [k]$ such that $(u,l)(v,l) \notin E(H)$. We call the matching $E_{H}(L(u),L(v))$, the \emph{twist}.

\begin{figure}[h]
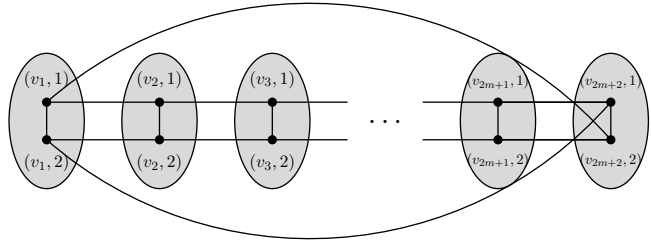

\centering
\begin{graph}

\coordinate (v11) at (-1.5,0);
        \coordinate (v12) at (-1.5,0.5);
        \coordinate (v21) at (0,0);
        \coordinate (v22) at (0,0.5);
        \coordinate (v31) at (1.5,0);
        \coordinate (v32) at (1.5,0.5);
        \coordinate (dot) at (3,0.25);
        \coordinate (v41) at (4.5,0);
        \coordinate (v42) at (4.5,0.5);
        \coordinate (v51) at (6,0);
        \coordinate (v52) at (6,0.5);
        
        \draw[fill=gray!30] (-1.5,0.25) ellipse (5mm and 9mm);
        \draw[fill=gray!30] (0,0.25) ellipse (5mm and 9mm);
        \draw[fill=gray!30] (1.5,0.25) ellipse (5mm and 9mm);
        
        \draw[fill=gray!30] (4.5,0.25) ellipse (5mm and 9mm);
        \draw[fill=gray!30] (6,0.25) ellipse (5mm and 9mm);
        
        \node at (dot) {$\hdots$};
        \draw[fill=black] (v11) circle[radius=1.5pt] node[below=3pt,scale=0.6] {$(v_1,2)$};
        \draw[fill=black] (v12) circle[radius=1.5pt] node[above=3pt,scale=0.6] {$(v_1,1)$};
        \draw[fill=black] (v21) circle[radius=1.5pt] node[below=3pt,scale=0.6] {$(v_2,2)$};
        \draw[fill=black] (v22) circle[radius=1.5pt] node[above=3pt,scale=0.6] {$(v_2,1)$};
        \draw[fill=black] (v31) circle[radius=1.5pt] node[below=3pt,scale=0.6] {$(v_3,2)$};
        \draw[fill=black] (v32) circle[radius=1.5pt] node[above=3pt,scale=0.6] {$(v_3,1)$};
        \draw[fill=black] (v42) circle[radius=1.5pt] node[above=3pt,scale=0.5] {$(v_{2m+1},1)$};
        \draw[fill=black] (v41) circle[radius=1.5pt] node[below=3pt,scale=0.5] {$(v_{2m+1},2)$};
        \draw[fill=black] (v52) circle[radius=1.5pt] node[above=3pt,scale=0.5] {$(v_{2m+2},1)$};
        \draw[fill=black] (v51) circle[radius=1.5pt] node[below=3pt,scale=0.5] {$(v_{2m+2},2)$};
        \draw (v11) -- (v12); \draw (v21) -- (v22); \draw (v31) -- (v32); \draw (v41) -- (v42); \draw (v51) -- (v52);
        
        \draw (v11) -- (v21); \draw (v21) -- (v31);
        \draw (v12) -- (v22); \draw (v22) -- (v32);
        \draw (v32) -- (2.5,0.5);
        \draw (v31) -- (2.5,0);
        \draw (3.5,0.5) -- (v52);
        \draw (3.5,0) -- (v51);
        \draw (v41) -- (v51);
        \draw (v42) -- (v52);
        \path[line width=0.5] (v11) edge[bend right=45] (v52);
        \path[line width=0.5] (v12) edge[bend right=-45] (v51);

\end{graph}
\caption{A $2$-fold cover of an even cycle with a twisted-canonical labeling}
\label{figure:TCL}
\end{figure}

The following result is a straightforward analogue of Proposition~\ref{prop: treecl} for twisted-canonical labeling.

\begin{pro} \label{prop: treetcl}
Let $T$ be a tree and let $\mathcal{H} = (L,H)$ be a full $m$-fold cover of $T$ where $m \in \mathbb{N}$. Then $\mathcal{H}$ has a twisted-canonical labeling.
\end{pro}

Notice that by Propositions~\ref{prop: treecl} and~\ref{prop: treetcl}, a cover that has a canonical labeling may also have a twisted-canonical labeling. In contrast, we have the following result for cycles.

\begin{lem} \label{lem: twofold}
Let $G$ be a cycle and suppose $\mathcal{H}$ is a $2$-fold cover of $G$. Then $\mathcal{H}$ has a twisted-canonical labeling if and only if it is full and has no canonical labeling.
\end{lem}
\begin{proof}
    Suppose $G$ is an $m$-cycle and the vertices of $G$ are ordered cyclically as $v_{1},\hdots,v_{m}$. Suppose $\mathcal{H} = (L,H)$. Let $H'$ be the spanning subgraph of $H$ such that the edge set of $H'$ only consists of the cross edges of $H$. Suppose $\mathcal{H}$ has a twisted-canonical labeling. Without loss of generality, suppose the vertices of $H$ are named so that $E_{H}(L(v_{m}),L(v_{1}))$ is the twist. Hence we know that for each $v_{i}v_{j}\in E(G) - \{v_{m}v_{1}\}$, $(v_{i},1)(v_{j},1)$ and $(v_{i},2)(v_{j},2)$ are edges in $H$. Additionally $(v_{m},1)(v_{1},2)$, $(v_{m},2)(v_{1},1) \in E(H)$. Notice that \\ $(v_{1},1)(v_{2},1), \hdots, (v_{m-1},1)(v_{m},1), (v_{m},1)(v_{1},2), (v_{1},2)(v_{2},2), \hdots, (v_{m-1},2)(v_{m},2)$ are edges of a spanning path in $H'$.  Now, for the sake of contradiction, suppose $\mathcal{H}$ has a canonical labeling. So we can rename the vertices of $H$ so that for each $u_{i}u_{j}\in E(G)$, $(u_{i},1)(u_{j},1)$ and $(u_{i},2)(u_{j},2)$ are edges in $H$ where $u_{i}$ and $v_{i}$ refer to the same vertex in $G$ for each $i \in [m]$. After naming the vertices of $H$ canonically, notice that all the edges in $H'$ are among the vertices having the same second coordinate. Hence $H'$ does not have a spanning path which is a contradiction.
    
Conversely, suppose $\mathcal{H}$ is full and has no canonical labeling. Let $uv \in E(G)$. The cover $(L, H - E_{H}(L(u),L(v)))$ of $G-uv$, has a canonical labeling by Proposition~\ref{prop: treecl}. Hence we can name the vertices of $H$ such that for each $xy\in E(G) - \{uv\}$, the edges $(x,1)(y,1)$ and $(x,2)(y,2)$ are in $E(H)$. Since we know that $\mathcal{H}$ does not have a canonical labeling, it must then be the case that $(u,1)(v,2)$ and $(u,2)(v,1)$ are in $E(H)$. Hence $\mathcal{H}$ has a twisted-canonical labeling.
\end{proof}
Using the notion of twisted-canonical labeling, we can characterize the bad 2-fold covers of even cycles. See~\cite{KO19} for a more general result about bad covers of non-DP-degree colorable graphs, which also implies the following result. Note that in~\cite{KO19}, a cover with twisted-canonical labeling is called a \emph{M\"{o}bius ladder}. 
\begin{lem} \label{lem: evenchar}
Suppose that $G$ is an even cycle and $\mathcal{H} = (L,H)$ is a cover of $G$ where $|L(v)| \geq 2$ for each $v \in V(G)$. Then $G$ does not admit an $\mathcal{H}$-coloring if and only if $\mathcal{H}$ is a $2$-fold cover with a twisted-canonical labeling.
\end{lem}

\begin{proof}
   Let $G = C_{2m+2}$ for some $m\in\mathbb{N}$. Suppose the vertices of $G$ are ordered cyclically as $x_{1},x_{2},\hdots,x_{2m+2}$.  Suppose $\mathcal{H}$ is a $2$-fold cover with a twisted-canonical labeling. We will show that $G$ does not admit an $\mathcal{H}$-coloring. Without loss of generality, suppose the vertices of $H$ are named so that $E_{H}(L(x_{1}),\allowbreak L(x_{2m+2}))$ is the twist. Notice that this means $(x_{1},\allowbreak 1)(x_{2m+2},\allowbreak 2)$ and $(x_{1},\allowbreak 2)(x_{2m+2},\allowbreak 1)$ are edges in $H$.  For the sake of contradiction, suppose $H$ contains an independent set $I$ of size $2m+2$. Without loss of generality, suppose $(x_1,1) \in I$. This means $I = \{(x_i,1) : i\text{ is odd},1\leq i\leq 2m + 1\}\cup \{(x_i,2) : i\text{ is even},2\leq i\leq 2m+2\}$.  Since $(x_1,1)$ and $(x_{2m+2},2)$ are adjacent in $H$, this is a contradiction. Hence $G$ does not admit an $\mathcal{H}$-coloring.
  
We will now prove the contrapositive of the converse.  Suppose first that $\mathcal{H}$ is a 2-fold cover with no twisted-canonical labeling. Clearly, we may assume that $\mathcal{H}$ is a full cover. By Lemma~\ref{lem: twofold}, $\mathcal{H}$ has a canonical labeling. Thus $G$ admits an $\mathcal{H}$-coloring.
    
Now suppose $\mathcal{H}$ is not 2-fold, then we know that there must exist some $u\in V(G)$ such that $|L(u)|\geq 3$. Without loss of generality, suppose $x_{2m+2} = u$. Then we can construct an $\mathcal{H}$-coloring $I$ greedily as in the proof of Lemma~\ref{lem: oddchar}.
\end{proof}

\subsection{Sharpness of Theorem~\ref{thm: cartprodcompbipartite}}
     
     Assuming the same setup as in the definition of volatile coloring in Section~\ref{sec:volatile}, we first prove an upper bound on the number of volatile $\mathcal{H}_{X}$-colorings for each $M_{y_{q}}$ when $M = C_{2m+2} \square K_{1,t}$. This is essential for proving the sharpness result.

     \begin{lem} \label{lem: evenvolatileupper}
     Let $M = C_{2m+2}\square K_{1,t}$ and let $\mathcal{H}$ be a $3$-fold cover of $M$. Then for each $q \in [t]$, $M_{y_{q}}$ has at most $1$ volatile $\mathcal{H}_{X}$-coloring.
     \end{lem}
     
     \begin{proof}
         Suppose the vertices of the even cycle are ordered cyclically as $v_{1},\hdots,v_{2m+2}$. Let $X = \{x\}$. Suppose $\mathcal{H} = (L,H)$ where $L(u,v) = \{(u,v,i) : i \in [3]\}$ for each $(u,v) \in V(M)$.  Suppose for $r \in [t]$, $M_{y_{r}}$ has at least 2 distinct volatile $\mathcal{H}_{X}$-colorings. Let $I_{1}$ and $I_{2}$ denote two of these colorings. For each $j \in [2]$, let $D_{j} = \{u \in V(H_{y_{r}}) : N_{H}(u)\cap I_{j} = \emptyset\}$ and $H^{j}_{y_{r}} = H[D_{j}]$. For each $u \in V(M_{y_{r}})$, we let $L^{j}_{y_{r}}(u) = L_{y_{r}}(u)\cap D_{j}$. Let $\mathcal{H}^{j}_{y_{r}} = (L^{j}_{y_{r}}, H^{j}_{y_{r}})$.
    
    Suppose $I_{1} \cap I_{2} \neq \emptyset$. Since $I_{1}$ and $I_{2}$ are distinct, there exist adjacent vertices $v_{i_{0}}, v_{i_{1}}$ in the even cycle such that $L(v_{i_{0}},x) \cap I_{1} \neq L(v_{i_{0}},x) \cap I_{2}$ and $L(v_{i_{1}},x) \cap I_{1} = L(v_{i_{1}},x) \cap I_{2}$. Without loss of generality, suppose $i_{0} = 1$ and $i_{1} = 2$. Let $L(v_{1},x) \cap I_{1} = \{(v_{1},x,l_{1})\}$, $L(v_{1},x) \cap I_{2} = \{(v_{1},x,l_{2})\}$, and $(v_{2},x,k_{1})$ be the vertex in $L(v_{2},x) \cap I_{1}$ and $L(v_{2},x) \cap I_{2}$. Notice that $l_{1}, l_{2}, k_{1} \in [3]$.
    
    Since $I_{1}$ and $I_{2}$ are volatile for $M_{y_{r}}$, there is no $\mathcal{H}^{1}_{y_{r}}$-coloring and  no $\mathcal{H}^{2}_{y_{r}}$-coloring of $M_{y_{r}}$. By Lemma~\ref{lem: evenchar}, $\mathcal{H}^{1}_{y_{r}}$ and $\mathcal{H}^{2}_{y_{r}}$ are $2$-fold covers with twisted-canonical labelings. Since $\mathcal{H}^{1}_{y_{r}}$ and $\mathcal{H}^{2}_{y_{r}}$ are $2$-fold, each of $(v_1, x, l_{1}), (v_1, x, l_{2})$ and $(v_{2},x,k_{1})$ have exactly one neighbor in $H_{y_{r}}$. Moreover, these neighbors are pairwise distinct. Let $(v_{1},y_{r},l_{3}), (v_{1},y_{r},l_{4})$, and $(v_{2},y_{r},k_{2})$ be the neighbors of $(v_1,x,l_{1})$, $(v_1,x,l_{2})$ and $(v_2,x,k_{1})$ in $H_{y_{r}}$ respectively. Note that $L^{1}_{y_{r}}(v_1, y_{r}) = L(v_1,y_{r}) - \{(v_{1},y_{r},l_{3})\}$, $L^{2}_{y_{r}}(v_1, y_{r}) = L(v_1,y_{r}) - \{(v_{1},y_{r},l_{4})\}$ and $L^{1}_{y_{r}}(v_2, y_{r}) = L^{2}_{y_{r}}(v_2, y_{r}) = L(v_2,y_{r}) - \{(v_{2},y_{r},k_{2})\}$. Therefore, $|L^{1}_{y_{r}}(v_1, y_{r}) \cap L^{2}_{y_{r}}(v_1, y_{r})| = 1$. We can let $(v_1,y_{r},l_{5})$ be the third element of $L_{y_{r}}(v_1, y_{r})$ so that $L^{1}_{y_{r}}(v_1, y_{r}) = \{(v_1,y_{r},l_{4}),\allowbreak (v_1,y_{r},l_{5})\}$, $L^{2}_{y_{r}}(v_1, y_{r}) = \{(v_1,y_{r},l_{3}),\allowbreak (v_1,y_{r},l_{5})\}$. Similarly, we can let $(v_2,y_{r},k_{3})$ and $(v_2,\allowbreak y_{r},\allowbreak k_{4})$ be the remaining two elements of $L_{y_{r}}(v_2, y_{r})$ so that $L^{1}_{y_{r}}(v_2, y_{r}) = L^{2}_{y_{r}}(v_2, y_{r}) = \{(v_2,y_{r},k_{3}),\allowbreak (v_2,y_{r},k_{4})\}$. Note that $\{l_{3}, l_{4},l_{5}\} = \{k_{2}, k_{3},k_{4}\} = [3]$. 

Since $\mathcal{H}^{1}_{y_{r}}$ and $\mathcal{H}^{2}_{y_{r}}$ have twisted-canonical labelings, the matchings $E_{H^{1}_{y_{r}}}(L^{1}_{y_{r}}(v_{1},y_{r}),\allowbreak L^{1}_{y_{r}}(v_{2},y_{r}))$ and $E_{H^{2}_{y_{r}}}(L^{2}_{y_{r}}(v_{1},y_{r}),L^{2}_{y_{r}}(v_{2},y_{r}))$ are perfect. Hence \\ $N_{H_{y_{r}}}(\{(v_1,y_{r},l_{4}),\allowbreak (v_1,y_{r},l_{5})\})\allowbreak = \{(v_2,y_{r},k_{3}), (v_2,y_{r},k_{4})\}$ and $(v_1,y_{r},l_{3})$ must have exactly one neighbor in $H_{y_{r}}$. This means that $(v_1,y_{r},l_{3})(v_2,y_{r},k_{2}) \in E(H_{y_{r}})$. On the other hand, $N_{H_{y_{r}}}(\{(v_1,y_{r},l_{3}), (v_1,y_{r},l_{5})\}) = \{(v_2,y_{r},k_{3}), (v_2,y_{r},k_{4})\}$. This implies $N_{H_{y_{r}}}((v_1,y_{r},l_{3})) \in \{(v_2,y_{r},k_{3}), (v_2,y_{r},k_{4})\}$. This is a contradiction.

Now suppose $I_{1} \cap I_{2} = \emptyset$. Since $\mathcal{H}_{y_{r}}^{1}$ is a $2$-fold cover of a cycle with a twisted-canonical labeling, it can not have a canonical labeling by Lemma~\ref{lem: twofold}. To complete the proof, we will inductively rename the vertices in $H^{1}_{y_{r}}$ to demonstrate that $\mathcal{H}_{y_{r}}^{1}$ has a canonical labeling.

Clearly $L(v_{1},x) \cap I_{1} \neq L(v_{1},x) \cap I_{2}$ and $L(v_{2},x) \cap I_{1} \neq L(v_{2},x) \cap I_{2}$. Let $L(v_{1},x) \cap I_{1} = \{(v_{1},x,l_{1})\}$, $L(v_{1},x) \cap I_{2} = \{(v_{1},x,l_{2})\}$, $L(v_{2},x) \cap I_{1} = \{(v_{2},x,k_{1})\}$ and $L(v_{2},x) \cap I_{2} = \{(v_{2},x,k_{2})\}$. Notice that $l_{1}, l_{2}, k_{1}, k_{2} \in [3]$. Let $(v_{1},y_{r},l_{3}), (v_{1},y_{r},l_{4})$, $(v_{2},y_{r},k_{3})$ and $(v_{2},y_{r},k_{4})$ be the neighbors of $(v_1, x, l_{1}), (v_1, x, l_{2})$, $(v_{2},x,k_{1})$ and $(v_{2},x,k_{2})$ in $H_{y_{r}}$ respectively. Since $I_{1}$ and $I_{2}$ are volatile for $M_{y_{r}}$, by Lemma~\ref{lem: evenchar}, the covers $\mathcal{H}^{1}_{y_{r}}$ and $\mathcal{H}^{2}_{y_{r}}$ are $2$-fold with twisted-canonical labelings. Clearly $l_{3} \neq l_{4}$ and $k_{3} \neq k_{4}$. We can let $(v_{1},y_{r},l_{5})$ and $(v_{2},y_{r},k_{5})$ be the remaining elements of $L_{y_{r}}(v_{1}, y_{r})$ and $L_{y_{r}}(v_{2}, y_{r})$ respectively. Notice that $L^{1}_{y_{r}}(v_{1}, y_{r}) = \{(v_{1},y_{r},l_{4}), (v_{1},y_{r},l_{5})\}$, $L^{2}_{y_{r}}(v_{1}, y_{r}) = \{(v_{1},y_{r},l_{3}), (v_{1},y_{r},l_{5})\}$, $L^{1}_{y_{r}}(v_{2}, y_{r}) = \{(v_{2},y_{r},k_{4}), (v_{2},y_{r},k_{5})\}$ and $L^{2}_{y_{r}}(v_{2}, y_{r}) = \{(v_{2},y_{r},k_{3}), (v_{2},y_{r},k_{5})\}$. Moreover note that $\{l_{3},\allowbreak l_{4},\allowbreak l_{5}\} = \{k_{3}, k_{4},k_{5}\} = [3]$. Since $\mathcal{H}^{1}_{y_{r}}$ and $\mathcal{H}^{2}_{y_{r}}$ have twisted-canonical labelings, the matchings $E_{H^{1}_{y_{r}}}(L^{1}_{y_{r}}(v_{1},y_{r}),L^{1}_{y_{r}}(v_{2},y_{r}))$ and $E_{H^{2}_{y_{r}}}(L^{2}_{y_{r}}(v_{1},y_{r}),L^{2}_{y_{r}}(v_{2},y_{r}))$ are perfect. This means that $(v_{1},y_{r},l_{3})(v_{2},y_{r},k_{3})$ and $(v_{1},y_{r},l_{4})(v_{2},y_{r},k_{4})$ must be edges in $H_{y_{r}}$. This further implies that $(v_{1},y_{r},l_{5})(v_{2},y_{r},k_{5})$ also must be an edge in $H_{y_{r}}$. In $L_{y_{r}}^{1}(v_{1},y_{r})$ and $L_{y_{r}}^{1}(v_{2},y_{r})$, we now rename the vertices $(v_{1},y_{r},l_{4})$, $(v_{1},y_{r},l_{5})$, $(v_{2},y_{r},k_{4})$ and $(v_{2},y_{r},k_{5})$ as $((v_{1},y_{r}),1)$, $((v_{1},y_{r}),2)$, $((v_{2},y_{r}),1)$ and $((v_{2},y_{r}),2)$ respectively. Notice that $((v_{1},y_{r}),j)((v_{2},y_{r}),j) \in E(H^{1}_{y_{r}})$ for each $j \in [2]$.

Now proceeding inductively for each $i \in \{3, \hdots, 2m+2\}$, we will rename the vertices in $L^{1}_{y_{r}}(v_{i},y_{r})$ as $((v_{i},y_{r}),1)$ and $((v_{i},y_{r}),2)$ so that $((v_{i-1},y_{r}),1)((v_{i},y_{r}),1)$ and $((v_{i-1},y_{r}),2)((v_{i},\allowbreak y_{r}),\allowbreak 2)$ are edges in $H_{y_{r}}$. By our renaming procedure, we have \\ $L^{1}_{y_{r}}(v_{i-1}, y_{r}) = \{((v_{i-1},y_{r}),1), ((v_{i-1},y_{r}),2)\}$ and $L^{2}_{y_{r}}(v_{i-1},\allowbreak y_{r}) = \{(v_{i-1},y_{r},l), ((v_{i-1},y_{r}),2)\}$ where $l \in [3]$. Suppose $L^{1}_{y_{r}}(v_{i},\allowbreak x)\cap I_{1} = (v_{i},x,p_{1})$ and $L^{1}_{y_{r}}(v_{i},x)\cap I_{2} = (v_{i},x,p_{2})$. Let $(v_{i},y_{r},p_{3})$ and $(v_{i},y_{r},p_{4})$ be the neighbors of $(v_{i},x,p_{1})$ and $(v_{i},x,p_{2})$ respectively in $H_{y_{r}}$. Let $(v_{i},y_{r},p_{5})$ be the remaining element of $L_{y_{r}}(v_{i},y_{r})$. We rename the vertices $(v_{i},y_{r},p_{4})$ and $(v_{i},y_{r},p_{5})$ as $((v_{i},y_{r}),1)$ and $((v_{i},y_{r}),2)$. Clearly $((v_{i-1},y_{r}),1)((v_{i},y_{r}),1)$ and \\ $((v_{i-1},y_{r}),2)((v_{i},y_{r}),2)$ must be edges in $H_{y_{r}}$.  This completes the renaming of the vertices in $H_{y_{r}}$. 

We will now show that this renaming demonstrates that $\mathcal{H}^{1}_{y_{r}}$ has a canonical labeling.  By our inductive procedure, we only need to verify that $((v_{2m+2},\allowbreak y_{r}),\allowbreak 1)((v_{1},\allowbreak y_{r}),\allowbreak 1)$ and $((v_{2m+2},\allowbreak y_{r}),\allowbreak 2)((v_{1},\allowbreak y_{r}),\allowbreak 2)$ are edges in $H^{1}_{y_{r}}$. We have $L^{1}_{y_{r}}(v_{2m+2},\allowbreak y_{r}) = \{((v_{2m+2},\allowbreak y_{r}),\allowbreak 1),\allowbreak ((v_{2m+2},\allowbreak y_{r}),\allowbreak 2)\}$. We can let $(v_{2m+2},\allowbreak y_{r},\allowbreak p)$ be the remaining vertex of $L^{2}_{y_{r}}(v_{2m+2},\allowbreak y_{r})$ so that $L^{2}_{y_{r}}(v_{2m+2},\allowbreak y_{r}) = \{(v_{2m+2},\allowbreak y_{r},\allowbreak p), ((v_{2m+2},\allowbreak y_{r}),\allowbreak 2)\}$. Notice that $p \in [3]$. Recall that $L^{1}_{y_{r}}(v_{1},\allowbreak y_{r}) = \{((v_{1},\allowbreak y_{r}),\allowbreak 1), ((v_{1},\allowbreak y_{r}),\allowbreak 2)\}$  and $L^{2}_{y_{r}}(v_{1},\allowbreak y_{r}) = \{(v_{1},\allowbreak y_{r},\allowbreak l_{3}), ((v_{1},\allowbreak y_{r}),\allowbreak 2)\}$ where $l_{3} \in [3]$. By Lemma~\ref{lem: evenchar}, the matchings $E_{H^{1}_{y_{r}}}(L^{1}_{y_{r}}(v_{2m+2},\allowbreak y_{r}),\allowbreak L^{1}_{y_{r}}(v_{1},\allowbreak y_{r}))$ and $E_{H^{2}_{y_{r}}}(L^{2}_{y_{r}}(v_{2m+2},\allowbreak y_{r}),\allowbreak L^{2}_{y_{r}}(v_{1},\allowbreak y_{r}))$ are perfect. This implies that $(v_{2m+2},\allowbreak y_{r},\allowbreak p)(v_{1},\allowbreak y_{r},\allowbreak l_{3}), ((v_{2m+2},\allowbreak y_{r}),\allowbreak 1)((v_{1},\allowbreak y_{r}),\allowbreak 1)$ and $((v_{2m+2},\allowbreak y_{r}),\allowbreak 2)((v_{1},\allowbreak y_{r}),\allowbreak 2)$ are edges in $H_{y_{r}}$. Hence $\mathcal{H}_{y_{r}}^{1}$ has a canonical labeling.
\end{proof} 
     
    We are now ready to prove the sharpness of Theorem~\ref{thm: cartprodcompbipartite}. 
      
      \begin{pro} \label{thm: evensharp}
     $\chi_{DP}(C_{2m + 2}\square K_{1,t}) = 4$ if and only if $t \geq P_{DP}(C_{2m+2},3) = 2^{2m+2}-1$.
     \end{pro}
     
     \begin{proof}
         Let $C=C_{2m+2}$, and let $K$ be the complete bipartite graph with the bipartition $X = \{x\}$ and $Y = \{y_{q} : q \in [t]\}$. Let $M = C\square K$.  Suppose $t \geq P_{DP}(C,3)$. We have $\chi_{DP}(M) = 4$ by Theorem~\ref{thm: cartprodcompbipartite}.
     
     Conversely, suppose $t < P_{DP}(C,3)$. Let $\mathcal{H} = (L,H)$ be an arbitrary $3$-fold cover of $M$ with $P_{DP}(M_X,\mathcal{H}_X) = d$. Clearly $d \geq P_{DP}(C,3) > t$. By Lemma~\ref{lem: evenvolatileupper}, for each $q \in [t]$, $M[V(C) \times y_{q}]$ has at most $1$ volatile $\mathcal{H}_{X}$-coloring. By Corollary~\ref{cor: hcoloring}, $M$ admits an $\mathcal{H}$-coloring. Thus $\chi_{DP}(M) \leq 3$.
\end{proof}
     
     We do not know of any other sharpness examples for Theorem~\ref{thm: cartprodcompbipartite} which leads us to believe that Theorem~\ref{thm: cartprodcompbipartite} is not sharp very often. This motivates us to study Question~\ref{ques: 2} further.

     \section{Cartesian Product of a Cycle and Complete Bipartite Graph} \label{four}
     In this Section, we will completely answer Question~\ref{ques: 2} when $G$ is an odd cycle and $k=1$ by showing that $f(C_{2m+1},1) = {P_{DP}(C_{2m + 1},3)}/{3}$ as compared to $f(C_{2m+2},1) = {P_{DP}(C_{2m + 2},3)}$ as given by Proposition~\ref{thm: evensharp}. More generally, in this Section, we make progress towards Question~\ref{ques: 2} by improving Theorem~\ref{thm: cartprodcompbipartite} when $G$ is a cycle and $k \geq 2$. We construct random covers with a combination of random matchings defined using an equivalence relation on an appropriate set of colorings, and matchings defined using canonical and twisted-canonical labelings. We show that there exists an appropriate bad cover by counting the expected number of volatile colorings and applying Lemma~\ref{lem: nohcoloring}. 
     
     \subsection{Odd Cycles}
We now explore the Cartesian product of an odd cycle and a complete bipartite graph. We start by defining an equivalence relation on the set of proper colorings of an odd cycle, which is the final ingredient we need to set up the process of creating a bad cover.

Let $G = C_{2m+1}$ with vertices ordered cyclically as $v_1,\hdots,v_{2m+1}$ and let $\mathcal{C}$ denote the set of all proper $k$-colorings of $G$. We define the equivalence relation $\sim$ on $\mathcal{C}$ such that if $c,d\in \mathcal{C}$, then $c\sim d$ if there exists a $j\in\mathbb{Z}_k$ such that $(c(v_i) - d(v_i))\mod k = j$ for all $i\in [2m+1]$.  The following Lemma is now immediate.

\begin{lem} \label{lem: oddequiclass}
Each equivalence class $\mathcal{E}$ of $\sim$ as defined above is of size $k$. Furthermore, if $c,d\in\mathcal{E}$ such that $c\neq d$, then $c(v_i)\neq d(v_i)$ for all $i\in[2m+1]$.
\end{lem}

Note that Lemma~\ref{lem: oddequiclass} gives a partition of the set of all proper $k$-colorings of an odd cycle into sets of size $k$ which immediately gives a corresponding partition of proper $\mathcal{H}$-colorings of an odd cycle into sets of size $k$ where $\mathcal{H}$ is a $k$-fold cover of an odd cycle with a canonical labeling. 

We are now ready to prove Theorem~\ref{thm: oddrandom}. We construct random $(k+2)$-fold covers of $C_{2m+1} \square K_{k,t}$ in such a way that, using the notation in the definition of volatile coloring in Section~\ref{sec:volatile}, each $\mathcal{H}_{X}$-coloring is volatile for some $M_{y_{q}}$. We use a combination of two types of matchings: (i) random matchings between $L(u,x_{j})$ and $L(u,y_{q})$ defined using the equivalence relation $\sim$ where $u \in V(G), x_{j} \in X$ and $y_{q} \in Y$, and (ii) matchings between $L(u,z)$ and $L(v,z)$ defined using a canonical labeling where $uv \in E(G)$ and $z \in X \cup Y$. We partition the set of all $\mathcal{H}_{X}$-colorings into sets of size $(k+2)^k$, and we find the expected number of volatile $\mathcal{H}_{X}$-colorings in each part. If that expectation is larger than $(k+2)^k-1$, we can use Lemma~\ref{lem: nohcoloring} to show a bad $(k+2)$-fold cover must exist.

     \begin{customthm} {\bf \ref{thm: oddrandom}}
        Given $k \in \mathbb{N}$, let $c_k = \left \lceil{\frac{k\ln(k+2)}{\ln{2}+(k-1)\ln(k+2)-\ln(2(k+2)^{k-1}-(k+1)!)}} \right \rceil$ if $k \geq 2$ and $c_1=1$. Then,\\ $\chi_{DP}(C_{2m+1}\square \allowbreak K_{k,t}) = k+3$ whenever $t\geq c_k\left(\frac{P_{DP}(C_{2m+1},k + 2)}{k+2}\right)^k = c_k\left(\frac{(k+1)^{2m+1} - (k+1)}{k+2}\right)^k$.
    \end{customthm}
    \begin{proof}
     For simplicity of notation, we will refer to $c_k$ as $c$ in the proof.
     
     Let $C$ be the odd cycle with vertices ordered cyclically as $u_{1},\hdots,u_{2m+1}$, and let $K$ be the complete bipartite graph with bipartition $X = \{x_{j} : j \in [k]\}$ and $Y = \{y_{q} : q \in [t]\}$. Let $M = C\square K$ and $M_{X} = M[\{(u_{i},x_{j}) : i \in [2m+1], j \in [k]\}]$.  By Theorem~\ref{thm: cartprod}, we have $\chi_{DP}(M) \leq \chi_{DP}(C) + \col(K) - 1 = k+3$. It remains to show that $\chi_{DP}(M) > k+2$. We form a $(k+2)$-fold cover $\mathcal{H} = (L,H)$ of $M$ using a partially random process.

For each $v \in V(K)$, let $\mathcal{H}_{v} = (L_{v}, H_{v})$ be a $(k+2)$-fold cover of $M[V(C)\times\{v\}]$ with a canonical labeling. For each $v \in V(K)$, we let $L(u,v) = L_{v}(u,v)$ for every $u \in V(C)$ and create edges so that $H[\bigcup_{u \in V(C)} L(u,v)] = H_{v}$. For simplicity, we rename each vertex $((u,v),l)$ in $V(H)$ as $(u,v,l)$. Next we use a random process to add matchings (possibly empty) between $L(u_{i},x_{j})$ and $L(u_{i},y_{q})$ for each $i \in [2m+1], j \in [k]$ and $q \in [t]$ to complete the construction of $\mathcal{H}$.

Let $\mathcal{H}_{X} = (L_{X},H_{X})$ denote the cover of $M_{X}$ where $L_{X}(u,x_{j}) = L_{x_{j}}(u,x_{j})$ for every $(u,x_{j}) \in V(C) \times X$ and $H_{X} = \bigcup_{j=1}^{k} H_{x_{j}}$. Let $\mathcal{C}$ and $\mathcal{I}$ denote the collection of all proper $(k+2)$-colorings and the collection of all $\mathcal{H}_{X}$-colorings respectively of $M_{X}$. Note that since $\mathcal{H}_{X}$ is a $(k+2)$-fold cover of $M_{X}$ with a canonical labeling, there exists a bijection between $\mathcal{C}$ and $\mathcal{I}$. Note that one such bijection is, $f: \mathcal{C} \rightarrow \mathcal{I}$ where $f(c) = \{(u_{i},x_{j},c(u_{i},x_{j})) : i\in [2m+1], j \in [k]\}$. Moreover, $P_{DP}(M_{X},k+2) = P(M_{X},k+2)$ (see Theorem 11 in~\cite{KM20}). Hence $P_{DP}(M_{X},\mathcal{H}_{X}) = P_{DP}(M_{X},k+2)$. Let $d = P(C,k+2) = (k+1)^{2m+1} - (k+1)$ so that $P_{DP}(M_{X},\mathcal{H}_{X}) = d^k$.

For every $j \in [k]$, let $\mathcal{C}_{j}$ and $\mathcal{I}_{j}$ denote the collection of all proper $(k+2)$-colorings and all $\mathcal{H}_{x_{j}}$-colorings respectively of $M[V(C)\times \{x_{j}\}]$. For each $j \in [k]$, $\mathcal{C}_{j}$ is partitioned into equivalence classes of size $(k+2)$ by Lemma~\ref{lem: oddequiclass} under the equivalence relation $\sim$ as described in the definition of $\sim$. We let $b = d/(k+2)$ and arbitrarily name these equivalence classes $\mathcal{E}_{j,p}$ where $p \in \left[b\right ]$.

We arbitrarily name the elements of the set $[b]^k$ as $\boldsymbol{p}_{1}, \ldots, \boldsymbol{p}_{b^k}$. Let $a \in [b^k]$ and suppose $\boldsymbol{p}_{a} = (p_{1}, p_{2}, \ldots, p_{k})$. We define $S_{\boldsymbol{p}_{a}} = \left\{\bigcup_{j=1}^{k}I_{p_{j}} : I_{p_{j}} \in f(\mathcal{E}_{j,p_{j}}) \text{ for each } j \in [k]\right\}$. Note that the size of $S_{\boldsymbol{p}_{a}}$ is $(k+2)^k$. Clearly $\left \{S_{\boldsymbol{p}_{a}} : a \in \left[b^k\right] \right \}$ is a partition of $\mathcal{I}$. For each $a \in [b^k]$, we associate $c$ cycles $M[V(C) \times \{y_{c(a-1)+1}\}], \ldots, M[V(C) \times \{y_{c(a-1)+c}\}]$ to $S_{\boldsymbol{p}_{a}}$. Note that this can be done since $t \geq cb^k$. By Lemma~\ref{lem: nohcoloring}, if there exists a cover $\mathcal{H^*}$ of $M$ such that for each $a \in \left[b^k\right]$, every $s \in S_{\boldsymbol{p}_{a}}$ is volatile for at least one of $M[V(C) \times \{y_{1}\}],\ldots, M[V(C) \times \{y_{t}\}]$ then $M$ does not have an $\mathcal{H^*}$-coloring. Next we use a probabilistic argument and show that there exists a way to create matchings between $L(u_{i},x_{j})$ and $L(u_{i},y_{q})$ for each $i \in [2m+1], j \in [k]$ and $q \in \{c(a-1)+1,\ldots, c(a-1)+c\}$ so that each $s \in S_{\boldsymbol{p}_{a}}$ is volatile for at least one of $M[V(C) \times \{y_{c(a-1)+1}\}],\ldots, M[V(C) \times \{y_{c(a-1)+c}\}]$.

Let $a \in [b^{k}]$ and suppose $\boldsymbol{p}_{a} = (p_{1}, p_{2}, \ldots, p_{k})$. Let $j \in [k]$ and suppose $f(\mathcal{E}_{j,p_{j}}) = \{I_{1},\ldots,I_{k+2}\}$. Note that by Lemma~\ref{lem: oddequiclass}, for each $j \in [k], f(\mathcal{E}_{j,p_{j}})$ is a partition of the vertex set of $H_{x_{j}}$. For each $\ell \in [c]$ and $j \in [k]$, we pick a random bijection $\sigma_{j}$ between $f(\mathcal{E}_{j,p_{j}})$ and $[k+2]$. Then we draw an edge between the vertex in $I_{l} \cap L(u_{i},x_{j})$ and the vertex $(u_{i},y_{q_{c(a-1)+\ell}},\sigma_{j}(I_{l}))$ for each $i \in [2m+1]$ and $l \in [k+2]$. This completes the construction of $\mathcal{H}$. It is easy to verify that $\mathcal{H}$ is a cover.

Let $a \in [b^{k}]$ and $s \in S_{\boldsymbol{p}_{a}}$. Suppose $\boldsymbol{p}_{a} = (p_{1}, p_{2}, \ldots, p_{k})$ and $s = \bigcup_{j=1}^{k}J_{p_{j}}$ where $J_{p_{j}} \in f(\mathcal{E}_{j,p_{j}})$ for each $j \in [k]$. For each $j \in [k]$, suppose the random bijection chosen previously between $f(\mathcal{E}_{j,p_{j}})$ and $[k+2]$ is $\sigma_{j}$. Let $r \in \{c(a-1)+1,\ldots, c(a-1)+c\}$. Let $D_{r} = \{w \in V(H_{y_{r}}) : N_{H}(w)\cap s = \emptyset\}$ and $H'_{y_{r}} = H[D_{r}]$. For each $u \in V(C)$, we define $L'_{y_{r}}(u,y_{r}) = L_{y_{r}}(u,y_{r})\cap D_{r}$. Let $\mathcal{H}'_{y_{r}} = (L'_{y_{r}}, H'_{y_{r}})$. By Lemma~\ref{lem: oddchar}, $s$ is volatile for $M[V(C) \times \{y_{r}\}]$ if and only if $\mathcal{H}'_{y_{r}}$ is a $2$-fold cover of $M[V(C) \times \{y_{r}\}]$ with a canonical labeling which happens if and only if the cardinality of the set $\{\sigma_{j}(J_{p_{j}}): j \in [k]\}$ is $k$.

We now calculate the probability that $s$ is volatile for $M[V(C) \times \{y_{r}\}]$. Since number of possible bijections between $f(\mathcal{E}_{j,p_{j}})$ and $[k+2]$ for each $j \in [k]$ is $(k+2)!$, there are total $((k+2)!)^k$ ways to add matchings in the way described earlier corresponding to each $S_{\boldsymbol{p}_{a}}$. We count the number of matchings that correspond to $s$ being volatile for $M[{V(C) \times y_{r}}]$ as follows. For $j=1$, there are $(k+2)$ possible values for $\sigma_{1}(J_{p_{1}})$. Then there are $(k+1)!$ possible bijections between $f(\mathcal{E}_{1,p_{1}})-\{J_{p_{1}}\}$ and $[k+2]-\{\sigma_{1}(J_{p_{1}})\}$. Note that if $k=1$, the probability that $s$ is volatile for $M[V(C) \times \{y_{r}\}]$ is $1$. If $k \geq 2$, consider $j=2$. Clearly $\sigma_{2}(J_{p_{2}})$ must be different from $\sigma_{1}(J_{p_{1}})$. Thus there are $k+1$ possible values for $\sigma_{2}(J_{p_{2}})$ and $(k+1)!$ possible bijections between $f(\mathcal{E}_{2,p_{2}})-\{J_{p_{2}}\}$ and $[k+2]-\{\sigma_{2}(J_{p_{2}})\}$. Continuing in this fashion, once we get to $j=k$, there are $(k+2)-(k-1)=3$ possible values for $\sigma_{k}(J_{p_{k}})$ and $(k+1)!$ possible bijections between $f(\mathcal{E}_{k,p_{k}})-\{J_{p_{k}}\}$ and $[k+2]-\{\sigma_{k}(J_{p_{k}})\}$. Thus for $k\geq2$, the probability that $s$ is volatile for $M[V(C) \times \{y_{r}\}]$ is 
\begin{equation*}
        \frac{(k+2)(k+1)!(k+1)(k+1)!\cdots3(k+1)!}{((k+2)!)^k} = \frac{(k+2)!}{2(k+2)^k}.
    \end{equation*}
 
Suppose $E_{a,s}$ is the event that $s$ is volatile for at least one of $M[V(C) \times \{y_{c(a-1)+1}\}],\allowbreak \ldots,\allowbreak M[V(C) \times \{y_{c(a-1)+c}\}]$.
 Hence $\mathbb{P}[E_{a,s}] = 1-\left( 1- {(k+2)!}/{2(k+2)^k} \right)^c$. Let $X_{a,s}$ be the indicator random variable such that $X_{a,s} = 1$ when $E_{a,s}$ occurs and $X_{a,s} = 0$ otherwise. Let $X_{a} = \Sigma_{s \in S_{\boldsymbol{p}_{a}}} X_{a,s}$. By linearity of expectation we have,
 \begin{equation*}
        \mathbb{E}[X_{a}] = (k+2)^k \left [1-\left( 1- \frac{(k+2)!}{2(k+2)^k} \right)^c \right ].
    \end{equation*}
    
Thus if $c$ satisfies
\begin{equation*}
        (k+2)^k \left [1-\left( 1- \frac{(k+2)!}{2(k+2)^k} \right)^c \right ] > (k+2)^k-1,
\end{equation*} then there exists a $(k+2)$-fold cover, $\mathcal{H}^{*}$, of $M$ such that for each $a \in [b^k]$ each $s \in S_{\boldsymbol{p}_{a}}$ is volatile for at least one of $M[V(C) \times \{y_{c(a-1)+1}\}],\ldots, M[V(C) \times \{y_{c(a-1)+c}\}]$.  We can show by a straightforward simplification that when $k \geq 2$ the above inequality holds if and only if $c > {k\ln(k+2)}/{(\ln{2}+(k-1)\ln(k+2)-\ln(2(k+2)^{k-1}-(k+1)!))}$.  It is also easy to see that the above inequality holds when $c=1$ and $k=1$.  Finally, note that by Lemma~\ref{lem: nohcoloring}, $M$ does not admit an $\mathcal{H}^{*}$-coloring and $k+2 < \chi_{DP}(M)$.  
\end{proof}

    To show the sharpness of Theorem~\ref{thm: oddrandom}, we need a lemma which utilizes the notation given in the definition of volatile coloring in Section~\ref{sec:volatile}.  

\begin{lem} \label{lem: oddvolatileupper}
     Let $M = C_{2m+1}\square K_{1,t}$ and let $\mathcal{H}$ be a $3$-fold cover of $M$. Then for each $q \in [t]$, $M_{y_{q}}$ has at most $3$ volatile $\mathcal{H}_{X}$-colorings.
     \end{lem}
     \begin{proof}
         Suppose the vertices of the odd cycle are ordered cyclically as $v_{1},\hdots,v_{2m+1}$. Let $X = \{x\}$. Let $\mathcal{H} = (L,H)$ where $L(u,v) = \{(u,v,i) : i \in [3]\}$ for each $(u,v) \in V(M)$. Let $\mathcal{I}$ be the set of all $\mathcal{H}_{X}$-colorings of $M_{X}$.
         
    Suppose for $r \in [t]$, $M_{y_{r}}$ has at least 4 distinct volatile $\mathcal{H}_{X}$-colorings. Let $J_{1}, J_{2}, J_{3}$ and $J_{4}$ denote $4$ of these colorings. By the pigeonhole principle, at least two of these colorings are non-disjoint. Let $I_{1}$ and $I_{2}$ denote two such colorings. Additionally since $I_{1}$ and $I_{2}$ are distinct, there exist $i_{0} \in [2m+1]$ and $i_{1} = (i_{0}\mod (2m+1)) + 1$ such that $L(v_{i_{0}},x) \cap I_{1} = L(v_{i_{0}},x) \cap I_{2} = \{(v_{i_{0}}, x, l_{0})\}$ and $L(v_{i_{1}},x) \cap I_{1} = \{(v_{i_{1}}, x, l_{1})\}$ and $L(v_{i_{1}},x) \cap I_{2} = \{(v_{i_{1}}, x, l_{2})\}$ where $l_{1} \neq l_{2}$ and $l_{0},l_{1},l_{2} \in [3]$. Without loss of generality, suppose $i_{0} = 1$ and suppose $i_{1} = 2$
    
   For each $j \in [2]$, let $D_{j} = \{u \in V(H_{y_{r}}) : N_{H}(u)\cap I_{j} = \emptyset\}$ and $H^{j}_{y_{r}} = H[D_{j}]$. For each $u \in V(M_{y_{r}})$, we define $L^{j}_{y_{r}}(u) = L_{y_{r}}(u)\cap D_{j}$. Let $\mathcal{H}^{j}_{y_{r}} = (L^{j}_{y_{r}}, H^{j}_{y_{r}})$.
    
Since $I_{1}$ and $I_{2}$ are volatile for $M_{y_{r}}$, there is no $\mathcal{H}^{1}_{y_{r}}$-coloring and  no $\mathcal{H}^{2}_{y_{r}}$-coloring of $M_{y_{r}}$. By Lemma~\ref{lem: oddchar}, $\mathcal{H}^{1}_{y_{r}}$ and $\mathcal{H}^{2}_{y_{r}}$ are $2$-fold covers with canonical labelings. Since $L(v_1,\allowbreak x) \cap I_{1} = L(v_1,\allowbreak x) \cap I_{2} = (v_1,\allowbreak x,\allowbreak l_{0})$ and $(v_1,\allowbreak x,\allowbreak l_{0})$ has a unique neighbor in $H_{y_{r}}$, we have $L^{1}_{y_{r}}(v_1,\allowbreak y_{r}) = L^{2}_{y_{r}}(v_1,\allowbreak y_{r}) = L(v_1,\allowbreak y_{r}) - N_{H}((v_1,\allowbreak x,\allowbreak l_{0}))$. Let $L(v_1,\allowbreak y_{r}) - N_{H}((v_1,\allowbreak x,\allowbreak l_{0})) = \{(v_1,\allowbreak y_{r},\allowbreak l_{3}),\allowbreak (v_1,\allowbreak y_{r},\allowbreak l_{4})\}$ where $l_{3},\allowbreak l_{4} \in [3]$. On the the other hand, since $l_{1} \neq l_{2}$,\allowbreak $(v_2,\allowbreak x,\allowbreak l_{1})$ and $(v_2,\allowbreak x,\allowbreak l_{2})$ have distinct neighbors in $H_{y_{r}}$. Therefore $|L^{1}_{y_{r}}(v_2,\allowbreak y_{r}) \cap L^{2}_{y_{r}}(v_2,\allowbreak y_{r})| = 1$. Let $(v_2,\allowbreak y_{r},\allowbreak l_{5}) \in L^{1}_{y_{r}}(v_2,\allowbreak y_{r}) - L^{2}_{y_{r}}(v_2,\allowbreak y_{r})$ and $(v_2,\allowbreak y_{r},\allowbreak l_{6}) \in L^{2}_{y_{r}}(v_2,\allowbreak y_{r}) - L^{1}_{y_{r}}(v_2,\allowbreak y_{r})$ and let $(v_2,\allowbreak y_{r},\allowbreak l_{7}) \in L^{1}_{y_{r}}(v_2,\allowbreak y_{r}) \cap L^{2}_{y_{r}}(v_2,\allowbreak y_{r})$ where $\{l_{5},\allowbreak l_{6},\allowbreak l_{7}\} = [3]$. To summarize, $L^{1}_{y_{r}}(v_1,\allowbreak y_{r}) = L^{2}_{y_{r}}(v_1,\allowbreak y_{r}) = \{(v_1,\allowbreak y_{r},\allowbreak l_{3}),\allowbreak (v_1,\allowbreak y_{r},\allowbreak l_{4})\}$, $L^{1}_{y_{r}}(v_2,\allowbreak y_{r}) = \{(v_2,\allowbreak y_{r},\allowbreak l_{5}),\allowbreak (v_2,\allowbreak y_{r},\allowbreak l_{7})\}$ and $L^{2}_{y_{r}}(v_2,\allowbreak y_{r}) = \{(v_2,\allowbreak y_{r},\allowbreak l_{6}),\allowbreak (v_2,\allowbreak y_{r},\allowbreak l_{7})\}$.

Since $\mathcal{H}^{1}_{y_{r}}$ has a canonical labeling, either $(v_1,\allowbreak y_{r},\allowbreak l_{3})(v_2,\allowbreak y_{r},\allowbreak l_{7}) \in E(H)$ or $(v_1,\allowbreak y_{r},\allowbreak l_{4})(v_2,\allowbreak y_{r},\allowbreak l_{7}) \in E(H)$. Without loss of generality, suppose $(v_1,\allowbreak y_{r},\allowbreak l_{3})(v_2,\allowbreak y_{r},\allowbreak l_{7}) \in E(H)$. This means $(v_1,\allowbreak y_{r},\allowbreak l_{4})(v_2,\allowbreak y_{r},\allowbreak l_{5}) \in E(H)$.

 Similarly since $\mathcal{H}^{2}_{y_{r}}$ has a canonical labeling, either $(v_1,\allowbreak y_{r},\allowbreak l_{3})(v_2,\allowbreak y_{r},\allowbreak l_{7}) \in E(H)$ or $(v_1,\allowbreak y_{r},\allowbreak l_{4})(v_2,\allowbreak y_{r},\allowbreak l_{7}) \in E(H)$. Since $E_{H}(L(v_1,\allowbreak y_{r}),\allowbreak L(v_2,\allowbreak y_{r}))$ is a matching and $(v_1,\allowbreak y_{r},\allowbreak l_{4})(v_2,\allowbreak y_{r},\allowbreak l_{5}) \in E(H)$, $(v_1,\allowbreak y_{r},\allowbreak l_{4})(v_2,\allowbreak y_{r},\allowbreak l_{7}) \notin E(H)$.  Hence we have $(v_1,\allowbreak y_{r},\allowbreak l_{3})(v_2,\allowbreak y_{r},\allowbreak l_{7}) \in E(H)$. Then $(v_1,\allowbreak y_{r},\allowbreak l_{4})(v_2,\allowbreak y_{r},\allowbreak l_{6}) \in E(H)$. Hence $E_{H}(L(v_1,\allowbreak y_{r}),\allowbreak L(v_2,\allowbreak y_{r}))$ is not a matching. Therefore $\mathcal{H}$ is not a cover and we have arrived at a contradiction.
     \end{proof}
     
     We are now ready to prove the sharpness result.
    
    \begin{pro} \label{cor: oddsharp}
         $\chi_{DP}(C_{2m + 1}\square K_{1,t}) = 4$ if and only if $t \geq \frac{P_{DP}(C_{2m + 1},3)}{3} = \frac{2^{2m+1}-2}{3}$.
     \end{pro}
     \begin{proof}
     Let $C = C_{2m+1}$ and $K$ be the complete bipartite graph with bipartition $X = \{x\}$ and $Y = \{y_{q} : q \in [t]\}$. Let $M = C\square K$ and $M_{x} = M[\{(u_{i},x) : i \in [2m+1]\}]$. Let $d = P_{DP}(C,3)$. Suppose $t \geq d/3$. Applying Theorem~\ref{thm: oddrandom} with $c_1=1$ implies $\chi_{DP}(M) = 4$.
     
     Conversely, suppose $t < {P_{DP}(C_{2m + 1},3)}/{3}$ or equivalently $d > 3t$. Let $\mathcal{H} = (L,H)$ be an arbitrary $3$-fold cover of $M$. By Lemma~\ref{lem: oddvolatileupper}, for each $q \in [t]$, $M[V(C) \times \{y_{q}\}]$ has at most $3$ volatile $\mathcal{H}_{X}$-colorings. By Corollary~\ref{cor: hcoloring}, $M$ admits an $\mathcal{H}$-coloring. Hence we have, $\chi_{DP}(M) \leq 3$.
\end{proof}
   
   \subsection{Even Cycles}  
     We next present a result that is an improvement on Theorem~\ref{thm: cartprodcompbipartite} when $G$ is an even cycle. Similar to the odd cycle case in the previous Subsection, we use an equivalence relation which is the final ingredient we need to set up the process of creating a bad cover of the Cartesian product of an even cycle and a complete bipartite graph. We start with the notion of a cover, $\mathcal{H}$, for which $P_{DP}(C_{2m},\mathcal{H}) = P_{DP}(C_{2m},k)$ for $m \geq 2$ and $k \in \N$.

     For any $m \geq 2$, a \emph{$k$-fold $C_{2m}$-twister} is a $k$-fold cover, $\mathcal{H} = (L,H)$, of $G = C_{2m}$ such that it is possible to order the vertices of $G$ cyclically as $u_1, \ldots, u_{2m}$, let $L(u_{i}) = \{(u_{i},l) : l \in [k]\}$ for each $i \in [2m]$ so that $\left(\bigcup_{l \in [k],i \in [2m-1]} \{(u_{i},l)(u_{i+1},l) \} \right) \cup \left(\bigcup_{l\in[k]} \{(u_{2m},l)(u_{1},(l\mod k)+1)\right)$ is the set of cross edges of $H$. 
\begin{figure}[h]
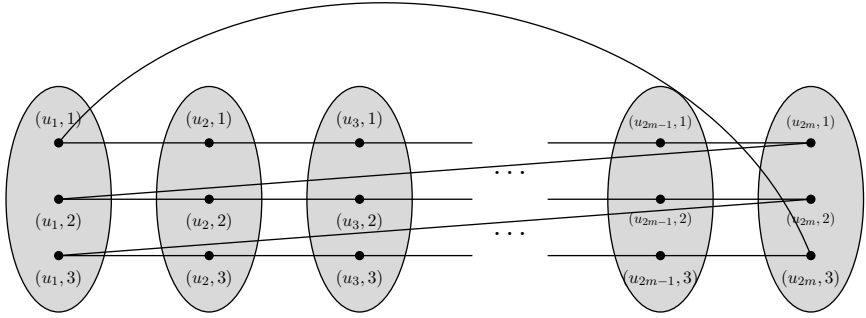

\centering
\begin{graph}
    \coordinate (v11) at (-2,-0.25);
        \coordinate (v12) at (-2,0.5);
        \coordinate (v13) at (-2,-1);
        \coordinate (v21) at (0,-0.25);
        \coordinate (v22) at (0,0.5);
        \coordinate (v23) at (0,-1);
        \coordinate (v31) at (2,-0.25);
        \coordinate (v32) at (2,0.5);
        \coordinate (v33) at (2,-1);
        \coordinate (dot1) at (4,0.1);
        \coordinate (dot2) at (4,-0.7);
        \coordinate (v41) at (6,-0.25);
        \coordinate (v42) at (6,0.5);
        \coordinate (v43) at (6,-1);
        \coordinate (v51) at (8,-0.25);
        \coordinate (v52) at (8,0.5);
        \coordinate (v53) at (8,-1);
        
        \draw[fill=gray!30] (-2,-0.25) ellipse (7mm and 15mm);
        \draw[fill=gray!30] (0,-0.25) ellipse (7mm and 15mm);
        \draw[fill=gray!30] (2,-0.25) ellipse (7mm and 15mm);
        
        \draw[fill=gray!30] (6,-0.25) ellipse (7mm and 15mm);
        \draw[fill=gray!30] (8,-0.25) ellipse (7mm and 15mm);
        
        \node at (dot1) {$\hdots$};
        \node at (dot2) {$\hdots$};
        \draw[fill=black] (v11) circle[radius=1.5pt] node[below=3pt,scale=0.6] {$(u_1,2)$};
        \draw[fill=black] (v12) circle[radius=1.5pt] node[above=3pt,scale=0.6] {$(u_1,1)$};
        \draw[fill=black] (v13) circle[radius=1.5pt] node[below=3pt,scale=0.6] {$(u_1,3)$};
        \draw[fill=black] (v21) circle[radius=1.5pt] node[below=3pt,scale=0.6] {$(u_2,2)$};
        \draw[fill=black] (v22) circle[radius=1.5pt] node[above=3pt,scale=0.6] {$(u_2,1)$};
        \draw[fill=black] (v23) circle[radius=1.5pt] node[below=3pt,scale=0.6] {$(u_2,3)$};
        \draw[fill=black] (v31) circle[radius=1.5pt] node[below=3pt,scale=0.6] {$(u_3,2)$};
        \draw[fill=black] (v32) circle[radius=1.5pt] node[above=3pt,scale=0.6] {$(u_3,1)$};
        \draw[fill=black] (v33) circle[radius=1.5pt] node[below=3pt,scale=0.6] {$(u_3,3)$};
        \draw[fill=black] (v42) circle[radius=1.5pt] node[above=3pt,scale=0.5] {$(u_{2m-1},1)$};
        \draw[fill=black] (v41) circle[radius=1.5pt] node[below=3pt,scale=0.5] {$(u_{2m-1},2)$};
        \draw[fill=black] (v43) circle[radius=1.5pt] node[below=3pt,scale=0.6] {$(u_{2m-1},3)$};
        \draw[fill=black] (v52) circle[radius=1.5pt] node[above=3pt,scale=0.5] {$(u_{2m},1)$};
        \draw[fill=black] (v51) circle[radius=1.5pt] node[below=3pt,scale=0.5] {$(u_{2m},2)$};
        \draw[fill=black] (v53) circle[radius=1.5pt] node[below=3pt,scale=0.6] {$(u_{2m},3)$};

        
        \draw (v11) -- (v21); \draw (v21) -- (v31);
        \draw (v12) -- (v22); \draw (v22) -- (v32);
        \draw (v13) -- (v23); \draw (v23) -- (v33);
        \draw (v32) -- (3.5,0.5);
        \draw (v31) -- (3.5,-0.25);
        \draw (v33) -- (3.5,-1);
        \draw (4.5,0.5) -- (v42);
        \draw (4.5,-0.25) -- (v41);
        \draw (4.5,-1) -- (v43);
        \draw (v41) -- (v51);
        \draw (v42) -- (v52);
        \draw (v43) -- (v53);
        \draw (v52) -- (v11);
        \draw (v51) -- (v13);
        \path[line width=0.5] (v53) edge[bend right=60] (v12);

\end{graph}
\caption{A $3$-fold $C_{2m}$-twister. Note that the set of vertices inside each oval is a clique; we have omitted drawing the corresponding edges for simplicity. }
\label{figure:twister}
\end{figure} 

By Lemma~23 in~\cite{KM20}, $P_{DP}(C_{2m},\mathcal{H}) = P_{DP}(C_{2m},k)$ when $\mathcal{H}$ is a $k$-fold $C_{2m}$-twister.

Let $G = C_{2m+2}$ with vertices ordered cyclically as $u_1, \ldots, u_{2m+2}$, and let $\mathcal{H} = (L,H)$ be a $k$-fold cover of $G$. Suppose for each $i \in [2m+2]$, $L(u_{i}) = \{(u_{i},l) : l \in [k]\}$. Let $\mathcal{I}$ denote the set of all $\mathcal{H}$-colorings of $G$. Suppose $I_{1},I_{2} \in \mathcal{I}$, and for each $i\in [2m+2]$, we let $c_{1,i}$ and $c_{2,i}$ be the second coordinate of the vertex in $I_{1} \cap L(u_{i})$ and the vertex in $I_{2} \cap L(u_{i})$ respectively. We define the relation $\approx$ on $\mathcal{I}$ such that $I_{1} \approx I_{2}$ if there exists a $j\in\mathbb{Z}_k$ such that $(c_{1,i}-c_{2,i})\mod k = j$ for all $i\in [2m+2]$. It is easy to see that $\approx$ is an equivalence relation.

\begin{lem} \label{lem: evenequiclass} Suppose $\mathcal{H}$ is a $k$-fold $C_{2m+2}$-twister of $G=C_{2m+2}$ where $k \geq 3$.  Using the notation of the definition of $\approx$ above, the following statements hold.

(i) Each equivalence class $\mathcal{E}$ of $\approx$ as defined above is of size $k$.

(ii) If $I_{1},I_{2}$ are in the same equivalence class of $\approx$ and $I_{1} \neq I_{2}$, then $c_{1,i} \neq c_{2,i}$ for all $i\in[2m+2]\}$.

\end{lem}
\begin{proof}
    We only prove Statement~(i) since Statement~(ii) is obvious.  Let $I\in \mathcal{E}$. For each $i\in [2m+2]$, we let $c_{i}$ be the second coordinate of the vertex in $I \cap L(u_{i})$. For each $l\in\mathbb{Z}_k, i\in [2m+2]$, we define $c_{l,i} = ((c_{i} - 1 + l) \mod k) + 1$. Then for each $l\in\mathbb{Z}_k$, we construct $I_{l}$ so that $I_{l} = \{(u_{i},c_{l,i}) : i \in [2m+2]\}$. Notice that $I_{0} = I$. Let $l'\in\mathbb{Z}_k$. Clearly $|I_{l'}\cap L(u_{i})| = 1$ for each $i \in [2m+2]$. We now show that $I_{l'}$ is an $\mathcal{H}$-coloring of $G$.  Suppose there exist two vertices in $I_{l'}$, $(u_{j},c_{l',j})$ and $(u_{j'},c_{l',j'})$ such that $(u_{j},c_{l',j})(u_{j'},c_{l',j'}) \in E(H)$. This implies that $(u_{j},c_{j}), (u_{j'},c_{j'}) \in I$. The fact that $\mathcal{H}$ is a twister of $G$ and our assumption $(u_{j},c_{l',j})(u_{j'},c_{l',j'}) \in E(H)$ imply that $(u_{j},c_{j})(u_{j'},c_{j'}) \in E(H)$. This is a contradiction. Thus $I_{l}$ is an $\mathcal{H}$-coloring of $G$  for each $l \in\mathbb{Z}_k$. Also, clearly $I_l\in\mathcal{E}$ for each $l\in\mathbb{Z}_k$. Therefore, $|\mathcal{E}|\geq k$.
    
    Now suppose $|\mathcal{E}| > k$. Let $J\in\mathcal{E}$ be an $\mathcal{H}$-coloring of $G$ that is not equal to $I_l$ for any $l\in\mathbb{Z}_k$. For each $i \in [2m+2]$, let $d_{i}$ be the second coordinate of the vertex in $J\cap L(u_{i})$. Then since $J\in\mathcal{E}$, there exists some $r\in\mathbb{Z}_k$ such that $(d_{i} - c_{i})\mod k = r$ for all $i\in[2m+2]$. Recall that $c_{i} = c_{0,i}$. Note that $(c_{r,i} - c_{i})\mod k = r$ for all $i\in[2m+2]$. Since $d_i$ and $c_{r,i}$ are both elements of $[k]$, $d_{i} = c_{r,i}$ for all $i\in[2m+2]$. Therefore $J = I_r$ and $|\mathcal{E}|\leq k$. This proves the Statement (i).
\end{proof}

We now prove Theorem~\ref{thm: evenrandom}.  The strategy of the proof is similar to that of Theorem~\ref{thm: oddrandom}.  We construct random $(k+2)$-fold covers of $C_{2m+2} \square K_{k,t}$ in such a way that, using the notation in the definition of volatile coloring in Section~\ref{sec:volatile}, each $\mathcal{H}_{X}$-coloring is volatile for some $M_{y_{q}}$. We begin so that $\mathcal{H}_{X}$ is built from $(k+2)$-fold $C_{2m+2}$-twisters. We then use a combination of two types of matchings: (i) random matchings between $L(u,x_{j})$ and $L(u,y_{q})$ defined using the equivalence relation $\approx$ where $u \in V(G), x_{j} \in X$, and $y_{q} \in Y$ and (ii) matchings between $L(u,y_{q})$ and $L(v,y_{q})$ defined using a twisted-canonical labeling where $uv \in E(G)$ and $y_{q} \in Y$. We partition the set of all $\mathcal{H}_{X}$-colorings into sets of size $(k+2)^k$, and we find the expected number of volatile $\mathcal{H}_{X}$-colorings in each part. If that expectation is larger than $(k+2)^k-1$, we can use Lemma~\ref{lem: nohcoloring} to show a bad $(k+2)$-fold cover must exist.

     \begin{customthm} {\bf \ref{thm: evenrandom}}
        Given $k \in \mathbb{N}$, let $c_k = \left \lceil{\frac{k\ln(k+2)}{k\ln(k+2)-\ln((k+2)^{k}-\floor{(k+2)/2}k!)}}\right \rceil$. Then $\chi_{DP}(C_{2m+2}\square \allowbreak K_{k,t})\allowbreak = \allowbreak k+3$ whenever $t\geq c_k\left(\frac{P_{DP}(C_{2m+2},k + 2)}{k+2}\right)^k = c_k\left(\frac{(k+1)^{2m+2}-1}{k+2}\right)^k$.
\end{customthm}
    
    \begin{proof}
     For simplicity of notation, we will refer to $c_k$ as $c$ in the proof.
     
     Let $C$ be the even cycle with vertices ordered cyclically as $u_{1},\hdots,u_{2m+2}$, and let $K$ be the complete bipartite graph with bipartition $X = \{x_{j} : j \in [k]\}$ and $Y = \{y_{q} : q \in [t]\}$. Let $M = C\square K$ and $M_{X} = M[\{(u_{i},x_{j}) : i \in [2m+2], j \in [k]\}]$.  By Theorem~\ref{thm: cartprod}, we have $\chi_{DP}(M) \leq \chi_{DP}(C) + \col(K) - 1 = k+3$. It remains to show that $\chi_{DP}(M) > k+2$. To show this, we form a $(k+2)$-fold cover $\mathcal{H} = (L,H)$ of $M$ using a partially random process.

For each $x \in X$, let $\mathcal{H}_{x} = (L_{x}, H_{x})$ be a $(k+2)$-fold $C_{2m+2}$-twister of $M[V(C)\times\{x\}]$. Clearly, $P_{DP}(M[V(C) \times \{x\}],\mathcal{H}_{x}) = P_{DP}(C,k+2)$. We then let $L(u,x) = L_{x}(u,x)$ for every $u \in V(C)$ and create edges in $H$ so that $H[\bigcup_{u \in V(C)} L(u,x)] = H_{x}$. Let $\mathcal{H}_{X} = (L_{X},H_{X})$ denote the cover of $M_{X}$ where $L_{X}(u,x) = L_{x}(u,x)$ for every $(u,x) \in V(C) \times X$ and $H_{X} = \bigcup_{x \in X} H_{x}$. For simplicity we name the vertices in $H$ so that for each $(u,v) \in V(M)$, $L(u,v) = \{(u,v,l) : l \in [k+2]\}$. Next we create edges so that for each $(u,v) \in V(C)\times Y$, $H[L(u,v)]$ is a clique on $k+2$ vertices. For each $j \in [t]$, we draw edges $(u_{i},y_{j},2l-1)(u_{i+1},y_{j},2l-1)$ and $(u_{i},y_{j},2l)(u_{i+1},y_{j},2l)$ for every $i \in [2m+1],l \in [\floor{(k+2)/2}]$. Next we draw edges $(u_{1},y_{j},2l-1)(u_{2m+2},y_{j},2l)$ and $(u_{1},y_{j},2l)(u_{2m+2},y_{j},2l-1)$ for each $j \in [t]$, $l \in [\floor{(k+2)/2}]$. Then, if $k+2$ is odd, we draw edges so that for each $j \in [t]$, $H[\{(u_{i},y_{j},k+2) : i \in [2m+2]\}]$ is a cycle with vertices ordered cyclically as $(u_{1},y_{j},k+2), \ldots, (u_{2m+2},y_{j},k+2)$.

Next we use a random process to add matchings(possibly empty) between $L(u_{i},x_{j})$ and $L(u_{i},y_{q})$ for each $i \in [2m+2], j \in [k]$ and $q \in [t]$ to complete the construction of $\mathcal{H}$.

 Let $\mathcal{I}$ denote the collection of all $\mathcal{H}_{X}$-colorings of $M_{X}$. Let $d = P_{DP}(C,k+2) = (k+1)^{2m+2}-1$ so that $P_{DP}(M_{X},\mathcal{H}_{X}) = d^k$. For every $j \in [k]$, let $\mathcal{I}_{j}$ denote the collection of all $\mathcal{H}_{x_{j}}$-colorings of $M[V(C)\times \{x_{j}\}]$. For each $j \in [k]$, $\mathcal{I}_{j}$ is partitioned into equivalence classes of size $k+2$ by Lemma~\ref{lem: evenequiclass} under the equivalence relation $\approx$ as described in the definition of $\approx$. We let $b = d/(k+2)$ and name these equivalence classes $\mathcal{E}_{j,p}$ where $p \in \left[b\right ]$. Notice that for each $j \in [t]$, $\{\mathcal{E}_{j,p} : p \in [b]\}$ is a partition of $\mathcal{I}_{j}$.
 
 We arbitrarily name the elements of the set $[b]^k$ as $\boldsymbol{p}_{1}, \ldots, \boldsymbol{p}_{b^k}$. Let $a \in [b^k]$ and suppose $\boldsymbol{p}_{a} = (p_{1}, p_{2}, \ldots, p_{k})$. We define $S_{\boldsymbol{p}_{a}} = \left\{\bigcup_{j=1}^{k}I_{p_{j}} : I_{p_{j}} \in \mathcal{E}_{j,p_{j}} \text{ for each } j \in [k]\right\}$. Note that the size of $S_{\boldsymbol{p}_{a}}$ is $(k+2)^k$. Clearly $\left \{S_{\boldsymbol{p}_{a}} : a \in \left[b^k\right] \right \}$ is a partition of $\mathcal{I}$. For each $a \in [b^k]$, we associate $c$ cycles $M[V(C) \times \{y_{c(a-1)+1}\}], \ldots, M[V(C) \times \{y_{c(a-1)+c}\}]$ to $S_{\boldsymbol{p}_{a}}$. Note that this can be done since $t \geq cb^k$. By Lemma~\ref{lem: nohcoloring}, if there exists a cover $\mathcal{H^*}$ of $M$ such that for each $a \in \left[b^k\right]$, every $s \in S_{\boldsymbol{p}_{a}}$ is volatile for at least one of $M[V(C) \times \{y_{1}\}],\ldots, M[V(C) \times \{y_{t}\}]$ then $M$ does not have an $\mathcal{H^*}$-coloring. Next we use a probabilistic argument and show that there exists a way to create matchings between $L(u_{i},x_{j})$ and $L(u_{i},y_{q})$ for each $i \in [2m+1], j \in [k]$ and $q \in \{c(a-1)+1,\ldots, c(a-1)+c\}$ so that each $s \in S_{\boldsymbol{p}_{a}}$ is volatile for at least one of $M[V(C) \times \{y_{c(a-1)+1}\}],\ldots, M[V(C) \times \{y_{c(a-1)+c}\}]$.

Let $a \in [b^{k}]$ and suppose $\boldsymbol{p}_{a} = (p_{1}, p_{2}, \ldots, p_{k})$. Let $j \in [k]$ and suppose $\mathcal{E}_{j,p_{j}} = \{I_{1},\ldots,I_{k+2}\}$. Note that by Lemma~\ref{lem: evenequiclass}, for each $j \in [k]$, $\mathcal{E}_{j,p_{j}}$ is a partition of the vertex set of $H_{x_{j}}$. For each $\ell \in [c]$ and $j \in [k]$, we pick a random bijection $\sigma_{j}$ between $\mathcal{E}_{j,p_{j}}$ and $[k+2]$. Then for each $l \in [k+2]$ and $i \in [2m+2]$, we draw an edge between the vertex in $I_{l} \cap L(u_{i},x_{j})$ and the vertex $(u_{i},y_{q_{c(a-1)+\ell}},\sigma_{j}(I_{l}))$. This completes the construction of $\mathcal{H}$. It is easy to verify that $\mathcal{H}$ is a cover.

Let $a \in [b^{k}]$ and $s \in S_{\boldsymbol{p}_{a}}$. Suppose $\boldsymbol{p}_{a} = (p_{1}, p_{2}, \ldots, p_{k})$ and $s = \bigcup_{j=1}^{k}J_{p_{j}}$ where $J_{p_{j}} \in \mathcal{E}_{j,p_{j}}$ for each $j \in [k]$. For each $j \in [k]$, suppose the random bijection chosen previously between $\mathcal{E}_{j,p_{j}}$ and $[k+2]$ is $\sigma_{j}$. Let $r \in \{c(a-1)+1,\ldots, c(a-1)+c\}$. Let $D_{r} = \{w \in V(H_{y_{r}}) : N_{H}(w)\cap s = \emptyset\}$ and $H'_{y_{r}} = H[D_{r}]$. For each $u \in V(C)$, we define $L'_{y_{r}}(u,y_{r}) = L(u,y_{r})\cap D_{r}$. Let $\mathcal{H}'_{y_{r}} = (L'_{y_{r}}, H'_{y_{r}})$. By Lemma~\ref{lem: evenchar}, $s$ is volatile for $M[V(C) \times \{y_{r}\}]$ if and only if $\mathcal{H}'_{y_{r}}$ is a $2$-fold cover of $M[V(C) \times \{y_{r}\}]$ with a twisted-canonical labeling. This happens if and only if the cardinality of the set $\{\sigma_{j}(J_{p_{j}}): j \in [k]\}$ is $k$ and there exists $l \in \left[\floor{(k+2)/2}\right]$ such that $H'_{y_{r}} = H[\bigcup_{i \in [2m+2],z \in [2]}(u_{i},y_{r},2l+z-2)]$ (i.e., $2l-1, 2l \notin \{\sigma_{j}(J_{p_{j}}): j \in [k]\}$).

 Since number of possible bijections between $f(\mathcal{E}_{j,p_{j}})$ and $[k+2]$ for each $j \in [k]$ is $(k+2)!$, there are total $((k+2)!)^k$ ways to add matchings in the way described earlier corresponding to each $S_{\boldsymbol{p}_{a}}$. We now calculate the probability that $s$ is volatile for $M[V(C) \times \{y_{r}\}]$. Since number of possible bijections between $\mathcal{E}_{j,p_{j}}$ and $[k+2]$ for each $j \in [k]$ is $(k+2)!$, there are total $((k+2)!)^k$ ways to add matchings in the way described earlier corresponding to each $S_{\boldsymbol{p}_{a}}$. We count the number of matchings that correspond to $s$ being volatile for $M[{V(C) \times y_{r}}]$ as follows. Recall that there are $\floor{(k+2)/2}$ possible values for $l$. Then for $j = 1$, there are $k$ possible values to choose from the set $[k+2]-\{2l-1,2l\}$ for $\sigma_{1}(J_{p_{1}})$ and $(k+1)!$ possible bijections between $\mathcal{E}_{1,p_{1}}$ and $[k+2]-\{\sigma_{1}(J_{p_{1}})\}$. If $k\geq2$ then for $j=2$, $\sigma_{2}(J_{p_{2}})$ must be different from $\sigma_{1}(J_{p_{1}}), 2l-1$ and $2l$. Thus there are $k-1$ possible values for $\sigma_{2}(J_{p_{2}})$ and $(k+1)!$ possible bijections between $\mathcal{E}_{2,p_{2}}-\{J_{p_{2}}\}$ and $[k+2]-\{\sigma_{2}(J_{p_{2}})\}$. Continuing in this fashion, once we get to $j=k$, there is $k-(k-1)=1$ possible value for $\sigma_{k}(J_{p_{k}})$ and $(k+1)!$ possible bijections between $\mathcal{E}_{k,p_{k}}-\{J_{p_{k}}\}$ and $[k+2]-\{\sigma_{k}(J_{p_{k}})\}$. Thus the probability that $s$ is volatile for $M[V(C) \times \{y_{r}\}]$ is 
  
  \begin{equation*}
        \frac{\floor{(k+2)/2}k(k+1)!(k-1)(k+1)!\cdots(k+1)!}{((k+2)!)^k} = \frac{\floor{(k+2)/2}k!}{(k+2)^k}.
    \end{equation*}
  
 Suppose $E_{a,s}$ is the event that $s$ is volatile for at least one of $M[V(C) \times \{y_{c(a-1)+1}\}],\allowbreak \ldots,\allowbreak M[V(C) \times \{y_{c(a-1)+c}\}]$.
 Hence $\mathbb{P}[E_{a,s}] = 1-\left( 1- {\floor{(k+2)/2}k!}/{(k+2)^k} \right)^c$. Let $X_{a,s}$ be the indicator random variable such that $X_{a,s} = 1$ when $E_{a,s}$ occurs and $X_{a,s} = 0$ otherwise. Let $X_{a} = \Sigma_{s \in S_{\boldsymbol{p}_{a}}} X_{a,s}$. By linearity of expectation we have,
 \begin{equation*}
        \mathbb{E}[X_{a}] = (k+2)^k \left [1-\left( 1- \frac{\floor{(k+2)/2}k!}{(k+2)^k} \right)^c \right ].
    \end{equation*}
    
Thus if $c$ satisfies 
\begin{equation*}
        (k+2)^k \left [1-\left( 1- \frac{\floor{(k+2)/2}k!}{(k+2)^k} \right)^c \right ] > (k+2)^k-1,
\end{equation*} then there exists a $(k+2)$-fold cover, $\mathcal{H}^{*}$, of $M$ such that for each $a \in [b^{k}]$ each $s \in S_{\boldsymbol{p}_{a}}$ is volatile for at least one of $M[V(C) \times \{y_{c(a-1)+1}\}],\ldots, M[V(C) \times \{y_{c(a-1)+c}\}]$.  We can show by a straightforward simplification that the above inequality holds if and only if $c > {k\ln(k+2)}/{(k\ln(k+2)-\ln((k+2)^{k}-\floor{(k+2)/2}k!))}$. Finally, note that by Lemma~\ref{lem: nohcoloring}, $M$ does not admit an $\mathcal{H}^{*}$-coloring and $k+2 < \chi_{DP}(M)$.
\end{proof}
     
Note that the sharpness of Theorem~\ref{thm: evenrandom} follows from Proposition~\ref{thm: evensharp}.

\end{document}